\documentclass[12pt]{amsart}
\usepackage[margin=1.2in]{geometry}
\title{Frames Induced by the Action of   Continuous Powers of an Operator}
\usepackage{hyperref}
\author[A. Aldroubi, L.X. Huang, A. Petrosyan]{A. Aldroubi, L.X. Huang, A. Petrosyan}

\date{\vspace{-5ex}}
\usepackage{amsthm, amsfonts}
\usepackage{ dsfont }
\usepackage{amssymb}
\usepackage{enumerate}
\usepackage{graphicx}
\usepackage{float}
\usepackage{bbm}
\usepackage{amsmath}
\usepackage{comment}
\usepackage{hyperref}
\usepackage{listings}
\usepackage{color}
\usepackage[dvipsnames]{xcolor}
\usepackage{enumitem}
\allowdisplaybreaks

\hypersetup{
    colorlinks,
    citecolor=blue,
    filecolor=blue,
    linkcolor=blue,
    urlcolor=blue
}

\newtheorem{theorem}{Theorem}[section]
\newtheorem{proposition}[theorem]{Proposition}
\newtheorem{lemma}[theorem]{Lemma}

\newtheorem{corollary}[theorem]{Corollary}
\newtheorem{definition}[theorem]{Definition}

\newtheorem{claim}[theorem]{Claim}
\newtheorem{conjecture}[theorem]{Conjecture}

\newtheorem{example}{Example}

\newcommand{\HH}{\mathcal{H}}

\newcommand{\R}{\mathbb{R}}
\newcommand{\Z}{\mathbb{Z}}
\newcommand{\N}{\mathbb{N}}
\newcommand{\CC}{\mathbb{C}}

\providecommand{\norm}[1]{\lVert#1\rVert}
\newcommand{\G}{\mathcal{G}}

\newcommand{\wi}{\widetilde}
\newcommand{\Om}{\Omega}
\newcommand{\rank}{{\rm rank\,}}

\newcommand {\la} {\langle}
\newcommand {\ra} {\rangle}
\begin{document}
	\date{}
\address{\textrm{(Akram Aldroubi)}
Department of Mathematics,
Vanderbilt University,
Nashville, Tennessee 37240-0001 USA}
\email{aldroubi@math.vanderbilt.edu}

\address{\textrm{(Longxiu Huang)}
	Department of Mathematics,
	Vanderbilt University,
	Nashville, Tennessee 37240-0001 USA}
\email{longxiu.huang@vanderbilt.edu}

\address{\textrm{(Armenak Petrosyan)}
Computational and Applied Mathematics Group,
Oak Ridge National Laboratory,
1 Bethel Valley Rd, Oak Ridge, Tennessee 37830, USA}
\email{petrosyana@ornl.gov}

\thanks{
The authors were supported in part by NSF Grant DMS- 1322099.
}


\maketitle
\begin{abstract}
	We investigate systems of the form $\{A^tg:g\in\G,t\in[0,L]\}$ where $A \in B(\HH)$ is a normal operator in a separable Hilbert space $\HH$, $\G\subset \HH$ is a countable set, and $L$ is a positive real number. Although the main goal of this work is to study the frame properties of  $\{A^tg:g\in\G,t\in[0,L]\}$, as  intermediate steps,  we explore the completeness and Bessel properties of such systems from a theoretical perspective, which are of interest by themselves. Beside the theoretical appeal of investigating such systems, their connections to dynamical and mobile sampling  make them fundamental for understanding and solving several major problems in   engineering and science.
\end{abstract}
\section{Introduction}
In a foundational paper, Duffin and Schaeffer  introduced the theory of frames  in the context of non-harmonic Fourier series \cite{DS52}. In this remarkable paper, the authors first gave conditions on a sequence of real numbers $\{\lambda_n\}_{n\in \Z}$ that induce  a Riesz basis of  exponentials $\{e^{i\lambda_n t}\}_{n\in\Z}$ for $L^2(- \frac 1 2, \frac 1 2)$. They then proposed the concept of frames which generalizes that of Riesz bases. Specifically, a frame $\{\phi_n\}_{n \in \Z}$ in a separable Hilbert space $\HH$ is a sequence of vectors satisfying
\begin{equation} 
\label {DSframes}
c\left\|f\right\|^2\le \sum\limits_{n\in \Z }\left\vert \la f,\phi_n\ra \right\vert^2\le C\|f\|^2, ~ \text {for all } f \in \HH,
\end {equation}
for some positive constants $c,C>0$. 
They showed  that, if \eqref {DSframes} holds, then (similar to a Riesz basis) any function $f\in \HH$ can be represented by the series
\[
f=\sum\limits_{n\in \Z }  \la f,\phi_n\ra  \tilde \phi_n,
\]
where $\{ \tilde \phi_n \}_{n\in \Z}$ is a dual frame and the convergence of the series is unconditional. Thus, every Riesz basis is a frame but a frame may have redundant vectors and hence need not be a basis.  However, the relation between Riesz bases and redundant frames is not self-evident. For example, there are frames for $\HH$  that have no  subsequences that are Riesz bases for $\HH$ (see, e.g., \cite {Christ16, Heil11} and the references therein).

The conditions on $\{\lambda_n\}_{n\in \Z}$ under which a system of exponentials $\{e^{i\lambda_n t}\}_{n\in\Z}$ becomes a frame for $L^2(- \frac 1 2, \frac 1 2)$ is also obtained in \cite{DS52}.  Using the Fourier transform, a set $\{e^{i\lambda_n t}\}_{n\in\Z}$ is a frame for $L^2(- \frac 1 2, \frac 1 2)$ if and only if  any function $f$ is in the Paley-Wiener space $PW_{1/2}=\{f \in L^2(\R): \hat f (\xi)=0   ~a.e.~\xi \notin (- \frac 1 2, \frac 1 2) \}$ can be recovered from its samples $\{f(\lambda_n)\}_{n\in \Z}$ in a {\em stable way}, i.e., there exists a bounded operator $R:\ell^2(\Z) \to \HH$  such that $R(f(\lambda_n))=f$.  This duality between reconstruction of functions from samples and frames has been used and extended  in many directions, including for wavelet representations, time-frequency analyses, and sampling in shift-invariant spaces (see, e.g., \cite {ABK08, AG01, CCK13, DS16, Daubechies_1992, Han09, Mallat09,  BM08, NO12, Sun14, ST17}).

\subsection{Dynamical sampling and frames induced by the action of  continuous powers of an operator}

\subsubsection {Dynamical sampling} 
The general problem in sampling theory is to reconstruct a function $f$ in a separable Hilbert space $\HH$ from  its samples. A natural idea is to sample the function $f$ at many  accessible positions  and one expects that, with some {\it a priori}  information, $f$ can be reconstructed from those samples. This idea is precisely the impetus of classical sampling theory. Related results can be found in, e.g.,  \cite{Adcock_2012,ABK08,AG01,BG12,Cand_s_2006,S01,Su06}. However, in real-world applications, there are many restrictions. For example, sampling may not be accessible at some required locations. Moreover, the spatial sampling density can be very limited, because  sensors are often expensive and it is costly to achieve a high sampling density.

In many instances, the functions  evolve over time by  a known driving operator. A common example is provided by diffusion and  modeled by the heat equation \cite{Lu_2009}. For such functions, a  novel theory has been developed recently, and it is termed {\it dynamical sampling theory}. The general idea of dynamical sampling is to reduce the spatial sampling density by increasing the temporal sampling rate. 

In dynamical sampling, the samples $\{(A^{t_j}f)(x_i): i\in \Z, j=0,\dots,J\}$ are taken repeatedly over time at some fixed spatial locations $X=\{x_i\}_{i\in \Z}$. Since the operator $A$ driving the evolution of $f$ can combine the information of $f$ from different locations, one may expect to recover the original function $f$ from $\{(A^{t_j}f)(x_i): i\in \Z, t_j\in T\}$, if the sampling locations are well chosen,  the operator $A$  is well-behaved, and the time-set $T=\{t_0,\dots, t_J\}$ (or $T=[0,L]$)  is large enough. The dynamical sampling problem is to derive necessary and sufficient conditions in terms of  the  operator $A$, the sampling set $X$, and  the sampling time-set $T$ such that the samples from different time levels are adequate to recover the original signal.

\subsubsection {Frames induced by the action of powers of an operator} 
The mathematical formulation of dynamical sampling  can be stated as follows. Let $A$ be a bounded linear operator on a separable Hilbert space $\HH$, and let $f\in\HH$ be  the initial state of an evolution system. At time $t$, the initial signal $f$ evolves to become
$$f_t=A^t f.$$ 
Given a countable (finite or countably infinite) set of vectors $\G \subset \HH$, the task is to find conditions on $A\in B(\HH)$, $\G$, and $T\subset [0, \infty)$ that allow the recovery or stable recovery of any function $f\in\HH$ from the set of samples 
\begin{equation} 
\left\{\langle A^tf,g\rangle: g\in\G, t\in T\right\}.
\end{equation}
By the recovery of $f$ we mean that there exists  an operator $R$ from $\G\times T$ to $\HH$ such that $R\big(\langle A^tf,g\rangle\big)=f$ for all $ f \in \HH$. While  by stable recovery of $f$ we mean that the operator $R$ is bounded. 
The problem above is equivalent to finding conditions on $A$, $\G$, and $T$ such that $\{A^{*t}g\}_{g\in\G,t\in T}$ (where $A^*$ denotes the adjoint of $A$) is complete or  a continuous frame for $\HH$, where the notion of continuous frames  generalizes that in \eqref {DSframes} \cite{Ali_1993,Ali_2000,Fornasier_2005,Gabardo_2003}. 
\begin{definition}
	Let $\HH$ be a complex Hilbert space and let $(\Omega, \mu)$ be a measure  space with positive measure $\mu$.   A mapping $F : \Omega \rightarrow \HH$ is called a  frame with respect to $(\Omega,\mu)$, if
	\begin{enumerate}[label=(\roman*)]
		\item F is weakly-measurable, i.e., $\omega \rightarrow \langle f, F(\omega)\rangle$ is a measurable
		function on $\Omega$  for all $ f\in\HH$;
		\item there exist constants $c$ and $C > 0$ such that
		\begin{equation}\label{ConFrame} 
		c\|f\|^2\leq \int_{\Omega}|\langle f,F(\omega)\rangle|^2d\mu(\omega)\leq C\|f\|^2,\text{ for all } f\in\HH.
		\end{equation}
	\end{enumerate}
	Here the constants $c$ and $C$ are called continuous frame (lower and upper) bounds. In addition, F is called a tight continuous frame if $c=C$.  The mapping F is called Bessel if the second inequality in \eqref{ConFrame} holds.
	In this case, C is called a Bessel constant.  
\end{definition}
The frame operator $S=S_{F}$ on $\HH$ associated with $F$ is defined in the weak sense by
$$S_{F}f=\int_{\Omega}\langle f,F(\omega)\rangle F(\omega)d\mu(\omega).$$
According to \eqref{ConFrame}, $S_F$ is well defined, invertible with bounded inverse (see \cite{Fornasier_2005}). %
Thus every $f\in\HH$ has the representations
\[f=S_{F}^{-1}S_F f=\int_{\Omega}\langle f,F(\omega)\rangle S_{F}^{-1}F(\omega)d\mu(\omega), \]
\[f=S_{F}S_F^{-1} f=\int_{\Omega}\langle f,S_F^{-1}F(\omega)\rangle F(\omega)d\mu(\omega). \]

If $\mu$ is the  counting measure and $\Omega=\mathbb{N}$, then one gets back the Duffin-Schaffer frame in \eqref {DSframes}.

%For this paper, 
In the sequel, 
$\Omega=\G\times[0,L]$, and $\mu$ is the product of the counting  measure on $\G$ and the Lebesgue measure on $[0,L]$. In this case,   $F$  is called a {\em semi-continuous frame} and \eqref {ConFrame} becomes
\begin {equation}\label {SCF}
c\|f\|^2\leq\sum\limits_{g\in\G}\int\limits_{0}^{L}|\la  f,A^tg\ra|^2 dt \leq C\|f\|^2,\text{ for all } f\in\HH.
\end{equation}
\subsection{Recent results on dynamical sampling and frames} 
Existing studies on various aspect of the dynamical sampling problem and related frame theory grew out of the  work in \cite{AADP13, ACCMP17, ACMT17, ADK13,  LuDV:11, RCLV11}, see, for example, \cite{AH17,  CMPP17,  CJS15,CH17, JT14,  JD15, JD17, Phi17,  ZLL17_2, ZLL17} and the references therein. However, except for a few, they all focus on  uniform discrete time-sets
$T\subseteq\{0,1,2,\ldots\}$, e.g., $T=\{1,\dots,N\}$ or $T=\N$ (see e.g., \cite {GRUV15}).

Even though the general dynamical sampling problem for discrete-time sets in finite dimensions (hence problems of systems and frames induced by iterations $\{A^ng:g \in \G, n \in T\}$) have been mostly  resolved  in \cite{ACMT17},  many problems and conjectures remain open for the infinite dimensional case. This state of affairs is not surprising because some of these problems take root in the  deep theory of functional analysis and operator theory such as the Kadison Singer Theorem  \cite {MSS15}, some open generalizations of the M\"untz-Sz\'asz Theorem \cite{Rudinrc}, and  the famous invariant subspace conjecture. 

When $T=\N$ and $A \in \mathcal B(\HH)$, it is not difficult to show that  
\begin{theorem}[\cite{AP17}]\label{noframeany}
	If, for an operator  $A\in B(\HH)$,  there exists a countable set of vectors $\G $  in $ \HH $ such that $\{A^n g\}_{g\in \G, \;n\geq 0}$ is  a frame in $\HH$, then   for every $f\in \HH$, $(A^*)^nf\to 0$ as $n\to \infty$. 		
\end{theorem}
Thus, in particular it is not possible to construct frames using non-negative iterations when $A$ is a unitary operator. For example, the right-shift operator $S$ on $\HH=\ell^2(\N)$ generates an orthonormal basis for $\ell^2(\N)$ by iterations over $\G=\{(1,0,\dots,)\}$. Clearly,  $(S^*)^nf\to 0$ as $n\to \infty$ for this case.  However, if we change the space to $\HH=\ell^2(\Z)$, the right-shift operator $S$ becomes unitary, and there exists no subset $\G$ of $\ell^2(\Z)$ such that $\{S^n g\}_{g\in \G, \;n\geq 0}$ is  a frame for $\ell^2(\Z)$.  

On the other hand, for normal operators, it is possible to find frames of the form  $\{A^n g\}_{g\in \G, \;n\geq 0}$; however, no such a frame can be a basis \cite{ACCMP17}. 

Frames for $\HH$ can  be generated by the iterative action on a single vector $g$, i.e., there exist normal operators and associated cyclic vectors such that $\{A^n g\}_{\;n\geq 0}$ is a frame for $\HH$  \cite{ACMT17}. Specifically,
\begin{theorem} [\cite{ACCMP17}]\label {OnePointFrame}
	Let $A$ be a bounded normal operator on an infinite dimensional Hilbert space $\HH$. Then,  $\{A^n g\}_{n\geq 0}$ is a frame for $\HH$ if and only if the following five conditions are satisfied:
	(i) $A=\sum_j\lambda_jP_j$, where $ P_j $ are rank one orthogonal projections; 
	(ii) $|\lambda_k| < 1$  for all $k$;  
	(iii) $|\lambda_k| \to 1$; 
	(iv) $\{\lambda_k\}$ satisfies Carleson's condition 
	$
	\inf_{n} \prod_{k\neq n} \frac{|\lambda_n-\lambda_k|}{|1-\bar{\lambda}_n\lambda_k|}\geq \delta,
	$
	for some $\delta>0$; 
	and 
	(v) $0<c\le \frac {\|P_jg\|} {\sqrt{1-|\lambda_k|^2}}  \le C< \infty$, for some constants $c,  C$.
\end{theorem}
It turns out that  if $A$ is normal in an infinite dimensional Hilbert space $\HH$, and $\{A^n g\}_{g\in \G, \;n\geq 0}$  is a frame  for some $ \G \subset \HH$ with $|\G| < \infty$, then $A$ is necessarily of the form described in Theorem \ref {OnePointFrame}:
\begin{theorem}[\cite{AP17}]\ \label{cor53}
	Let $A$ be a bounded normal operator in an infinite dimensional Hilbert space $\HH$.  If  the system of vectors  $\{A^n g\}_{g\in \G, \;n\geq 0}$  is a frame for some $ \G \subset \HH$ with $|\G| < \infty$, then  $A=\sum_j\lambda_jP_j$ where $ P_j $ are projections such that $\rank(P_j)\leq |\G|$ $\left(\text{i.e., the global multiplicity of $A$ is less than or equal to $ |\G| $}\right).$ In addition, (ii) and (iii) of Theorem \ref {OnePointFrame} are satisfied.
\end{theorem}
The necessary and sufficient conditions generalizing Theorem \ref {OnePointFrame} for the case $1<|G|<\infty$ have been derived in \cite {CMPP17}. 

Viewing Theorem \ref {OnePointFrame} from a different perspective, Christensen and Hasannasab ask whether a frame $\{h_n\}_{n \in I}$ has a representation of the form $h_n=A^nh_0$ for  some operator $A$ when $I=\N\cup \{0\}$ or $I=\Z$. This question is partially answered in  \cite{CH17OpRep} and gives rise to many new open problems and conjectures \cite{CH17}.

The set of self-adjoint operators is an important class of normal operators because it is often encountered in applications.  For this class, one can rule out  certain types of normalized frames.
\begin{theorem} [\cite{ACMT17}] \label{normframe}
	If $A$ is a self-adjoint operator on $\HH$, then the system  $\left\{\frac{A^n g}{\norm{A^n g}}\right\}_{g \in \G,\;n\geq 0}$ is not a frame for $\HH$.
\end{theorem}
However, for normal operators, the following conjecture remains open:
\begin {conjecture} \label {normframeConj} The statement of Theorem \ref {normframe} holds  for    normal operators.
\end {conjecture}  
Conjecture \ref {normframeConj} does not hold if the operator is not normal.  For example, the shift-operator $S$ on $\ell^2(\N)$ defined by $S(x_1,x_2,\dots)=(0,x_1,x_2,\dots)$, is not normal, and $\{S^ne_1\}$  is an orthonormal basis for $\ell^2(\N)$, where $e_1=(1,0,\dots )$.

\subsection{Contributions and Organization}  
The present work concentrates on  systems of the form $\{A^tg:g\in\G,t\in[0,L]\} \subseteq \HH$,  where $A\in\mathcal B(\HH)$. The goal is to study the frame property of such systems. To this end, we need to derive some other properties in the intermediate steps. In particular, we study the completeness and Besselness of these systems. 

For the completeness of $\{A^tg:g\in\G,t\in[0,L]\}$,  necessary and sufficient conditions are derived in Section \ref{completeness section}. In light of the results in \cite{ACCMP17}, the form of the necessary and sufficient conditions are not surprising. However, the proofs and reductions to the known cases are appealing due to the use of certain techniques of complex analysis, and they are useful for the analysis of frames in the subsequent sections. 

The Bessel  property  of the system  $\{A^tg:g\in\G,t\in[0,L]\}$  is investigated in  Section \ref{bessel section}. Specifically, if $\HH$ is a finite dimensional space (e.g., $\CC^d$) and $A$ is a normal operator in $\HH$, then the system $\{A^tg\}_{g\in\G, t\in[0,L]}$ being Bessel is equivalent to the Besselness of $\G$ in the space $Range(A)$.  On the other hand, if $\HH$ is an infinite dimensional separable Hilbert space and $A$ is a bounded invertible normal operator, then the only condition ensuring that $\{A^tg\}_{g\in\G,t\in[0,L]}$ is Bessel is that $\G$ itself is a Bessel system in $\HH$. In addition, an example is described to explain that the non-singularity of $A$ is necessary for the equivalence between the Besselness of $\{A^tg\}_{g\in\G,t\in[0,L]}$ and that of $\G$. 

Section 5 is devoted to the relations between a semi-continuous frame $\{A^tg\}_{g\in\G,t\in[0,L]}$ generated by the action of an operator $A\in \mathcal{B}(\HH)$  and the discrete systems generated by  its  time discretization.  Specifically, we show that under some mild conditions, $\{A^{t}g\}_{g\in\G,t\in [0,L]}$ is a semi-continuous frame if and only if  there exists  $T=\{t_i:i=I\}\subseteq[0,L)$ with $|I|<\infty$ such that $\{A^{t}g\}_{g\in\G,t\in T}$ is a frame system in $\HH$. Additionally, Theorem \ref{SCFrSA} shows that under proper conditions, the property that $\{A^t g\}_{g\in\G, t\in [0,L]}$ is a semi-continuous frame is independent of $L$.

\section{Notation and preliminaries}\label{sec2}

\subsection{Normal operators}
Let $\mathcal{B}(\HH)$ denote the space of bounded linear operators on  a complex  separable Hilbert space $\HH$. In the sequel, all the operators are assumed to be normal. Normal operators have the following invertibility property (see \cite[Theorem 12.12]{rudinfa91}).

\begin{theorem}
	If $A\in\mathcal{B}(\HH)$, then $A$ is invertible $\left(\text{i.e., $A$ has bounded inverse}\right)$ if and only if there exists $c>0$ such that $\|Af\|\geq c\|f\|$ for all $f\in\HH$. 
\end{theorem}

For completeness, the spectral theorem with multiplicity is stated below, and the following notation is used in its statement. 

For a  non-negative regular Borel measure $\mu$ on $\mathbb{C}$, $N_{\mu}$  will denote the multiplication operator acting on  $ L^2(\mu)$, i.e., for a $\mu$-measurable function $f:\CC\to \CC$ such that $\int_\CC |f(z)|^2d\mu(z)< \infty$, 
$$N_{\mu}f(z)=zf(z).$$

We will use the notation $[\mu]=[\nu]$ to denote two mutually absolutely continuous measures $\mu$ and $\nu$.

The operator $N_{\mu}^{(k)}$ will denote the direct sum of $k$ copies of $N_{\mu}$, i.e.,
\begin{equation*}
(N_\mu)^{(k)}=\oplus_{i=1}^{k}N_{\mu}.
\end{equation*}
Similarly, the space $(L^2(\mu))^{(k)}$ will denote the direct sum of $k$ copies of $L^2(\mu)$. %i.e.,
%\begin{equation*}
%	(L^2(\mu))^{(k)}=\bigoplus_{i=1}^{k}L^2(\mu).
%\end{equation*}

%For a Borel non-negative measure $\mu$, we will denote by $[\mu]$ the class of Borel measures that are mutually absolutely continuous with  $\mu$.

\begin{theorem}[\textbf{Spectral theorem with multiplicity}]\label{spectral theorem} For any normal operator $A$ on $\HH$ there are mutually singular  non-negative  Borel measures $\mu_j, 1\leq j\leq\infty$, such that $A$ is equivalent to the operator
	$$N_{\mu_{\infty}}^{(\infty)}\oplus N_{\mu_1}
	\oplus N_{\mu_2}^{(2)}\oplus\ldots,$$
	i.e., there exists a unitary transformation
	$$U:\HH\rightarrow(L^2(\mu_{\infty}))^{(\infty)}\oplus L^2(\mu_1)\oplus (L^2(\mu_2))^{(2)}\oplus\ldots$$ such that
	\begin{equation}\label{representation of normal}
	UAU^{-1}=N_{\mu_{\infty}}^{(\infty)}\oplus N_{\mu_1}\oplus N_{\mu_2}^{(2)}\oplus\ldots.
	\end{equation}
	Moreover, if $\tilde{A}$ is another normal operator with corresponding measures $\nu_{\infty},\nu_1,\nu_2,\ldots$, then $\tilde{A}$ is unitary equivalent to A if and only if $[\nu_j]=[\mu_j]$ for $j=1,\ldots,\infty.$
\end{theorem}
A proof of the theorem can be found in \cite[Ch. IX, Theorem 10.16]{conway} and \cite[Theorem 9.14]{conway1}. 

Since the measures $\mu_j$ are mutually singular, there are mutually disjoint  Borel sets $\{\mathcal{E}_j\}_{j=1}^{\infty}\cup \{\mathcal{E}_\infty\}$ such that $\mu_j$ is supported on $\mathcal{E}_j$ for every $1\leq j\leq\infty$.
The scalar-valued spectral measure $\mu$ associated with the normal operator $A$ is defined as
\begin{equation}\label{scalar spectral measure}
\mu=\sum_{1\leq j\leq\infty}\mu_j.
\end{equation}

The Borel function $m_{A}:\mathbb{C}\rightarrow \mathbb{N}^*\cup\{0\}$ given by
\begin{equation}
m_{A}(z)=\infty\cdot\chi_{\mathcal{E}_{\infty}}(z)+\sum_{j=1}^{\infty}j\chi_{\mathcal{E}_j}(z)
\end{equation}
is called the multiplicity function of the operator $A$, where $\mathbb{N}$ is the set of natural numbers starting with $1$, $\mathbb{N}^*=\mathbb{N}\cup\{\infty\}$, $\chi_{E}(z)$ is the characteristic function on set $E$  defined by  $\chi_{E}(z)=\begin{cases}
1,z\in E\\
0, \text{otherwise}
\end{cases}$ and $\infty\cdot 0=0.$

From Theorem \ref{spectral theorem}, every normal  operator is uniquely determined, up to a unitary equivalence,  by the pair $([\mu],m_A)$. 

For $j\in \N$, let $\Om_j$  be the set $ \{1,...,j\}$  and let $\Om_\infty$  be the set $\N$. Then $\ell^2(\Om_j) \cong \CC^j$, for $j\in \N$, and $\ell^2(\Omega_\infty) = \ell^2(\N).$ For $j=0$, we use $\ell^2(\Omega_0)$ to represent the trivial space $\{0\}$.

Let $\mathcal{W}$ be the Hilbert space 
$$\mathcal{W} =(L^2(\mu_\infty))^{(\infty)}\oplus L^2(\mu_1)\oplus (L^2(\mu_2))^{(2)}\oplus\cdots$$ 
associated with the operator $A$ and
let $U:\HH \rightarrow \mathcal{W}$ be the unitary operator given by Theorem~\ref{spectral theorem}.  If $g\in \HH$, we denote by $\wi{g}$
the image of $g$ under $U$. Since $\wi{g} \in \mathcal{W}$, one has   $\wi{g} = (\wi{g}_j)_{j\in \N^*}$, where 
$\wi{g}_j$ is the restriction of $\wi g$ to $(L^2(\mu_j))^{(j)}$. Thus, for any $j\in \N^*$, $ \wi{g}_j $ is a function from $ \CC $ to $ \ell^2(\Omega_j) $ and
$$\sum_{j\in \N^*} \quad \int_\CC \|\wi{g}_j(z)\|_{\ell^2(\Om_j)}^2d\mu_j(z) =\|g\|^2<\infty .$$
Let $P_j$ be the projection defined for every  $\wi g \in \mathcal W$  by $P_j\wi g=\wi f$, where $\wi f_j=\wi g_j$ and $\wi f_k=0$ for $k\ne j$.

Let $E$ be the spectral measure for the normal operator $A$. Then, for every $\mu$-measurable set $G\subseteq \CC$ and vectors $f,g$ in $\HH$, one has the following formula
$$\la E(G)f,g\ra _{\HH}\; = \int_G\left[\sum_{1\leq j\leq\infty} \chi_{\mathcal{E}_j}(z)\la \wi f_j(z),\wi g_j(z)\ra _{\ell^2(\Om_j)}\right]d\mu(z),$$
which relates the spectral measure of $A$ to the scalar-valued spectral measure of $A$.
\begin{definition}
	Given a normal operator $A$, $A^t$ is defined as follows: $$A^t:\HH\rightarrow\HH$$ by
	$$\langle A^tf_1, f_2\rangle=\int_{z\in\sigma(A)}z^t\langle \tilde{f}_1(z),\tilde{f}_2(z)\rangle d\mu(z), \text{ for all }f_1,f_2\in\HH,$$ where $z^t=\exp(t(\ln(|z|)+i\arg(z)))$ and $\arg(z)\in[-\pi,\pi).$
	
	Using the fact that $\exp(i\arg(z)+i\arg(\bar z))=1,$ it follows that $(A^*)^t=(A^t)^*$ for $t\in\mathbb{R}$. 
	
\end{definition}
Section  \ref{completeness section} will exploit the reductive operators which were introduced by P.Halmos and J.Wermer \cite{halmos,wermer}. For clarity, the definition is given as follows. 
\begin{definition}
	A closed subspace $V\subseteq\HH$ is called reducing for the operator $A$ if both $V$ and its orthogonal complement $V^{\perp}$ are invariant subspaces of $A$. 
\end{definition}
\begin{definition}
	An	operator $A$ is called reductive if every invariant subspace of $A$ is reducing.
\end{definition}
\subsection{Holomorphic Function}
The techniques of complex analysis, e.g., the properties of holomorphic functions (see \cite{Conway_1973,Rudinrc} and the references therein), are used extensively in the present work, including Montel's Theorem as stated below.
\begin{definition}[\textbf{Normal family}]
	A family $\mathfrak{F}$ of holomorphic functions in a region $X$ of the complex plane with values in $\mathbb{C}$ %\subseteq \mathbb{C}$ $\left(\text{with complex-valued function}\right)$
	is called normal if every sequence in $\mathfrak{F}$ contains a subsequence which converges uniformly  to a holomorphic function on compact subsets of X. 
\end{definition}
\begin{theorem}[\textbf{Montel's Theorem}] \label{Montel's theorem} A uniformly bounded family of holomorphic functions defined on an open subset of  the complex numbers is normal.
\end{theorem}

\section{Completeness}\label{completeness section}
In this section, we characterize  the completeness of the system $\{A^tg\}_{g\in\G,t\in[0,L]}$, where $A$ is a (reductive) normal operator on a  separable Hilbert space $\HH$,  $\G$ is a set of vectors in $\HH$, and $L$ is a finite positive number. 
\begin{theorem}\label{theorem 3.1}
	Let $A\in\mathcal{B}(\HH)$ be a normal operator,  
	and let $\G$ be a countable set of vectors in $\HH$ such that $\{A^tg\}_{g\in\G,t\in [0,L]}$ is complete in $\HH$. Let $\mu_{\infty},\mu_1,\mu_2,\ldots$ be the measures in the representation \eqref{representation of normal} of the operator $A$. Then for every $1\leq j\leq\infty$ and $\mu_j$-a.e. $z$, the  system of  vectors $\{\tilde{g}_j(z)\}_{g\in\G}$ is complete in $\ell^2(\Omega_j)$.
	
	If $A$ is also reductive, then $\{A^tg\}_{g\in\G,t\in [0,L]}$ being complete in $\HH$ is equivalent to $\{\tilde{g}_j(z)\}_{g\in\G}$ being complete in $\ell^2(\Omega_j)$ $\mu_j$-a.e. $z$ for every $1\leq j\leq\infty$.
\end{theorem}
Particularly, if the evolution operator belongs to the following class $\mathcal{A}$ of bounded self-adjoint operators:
\begin{eqnarray}
\mathcal{A}&=&\{A\in\mathcal{B}(\ell^2(\mathbb{N})): A=A^*, \nonumber\\
&&\text{and there exists a basis of $\ell^2(\mathbb{N})$ of eigenvectors of $A$}\},\label{equationA*}
\end{eqnarray}
then, for $A\in\mathcal{A}$, there exists a unitary operator $U$ such that $A=U^*DU$ with $D=\sum_{j}\lambda_j P_j$, where $\lambda_j$ are the spectrum of $A$ and $P_j$ is the orthogonal projection to the eigenspace $E_j$ of $D$  associated with the eigenvalue $\lambda_j$. Since  the operators in $\mathcal{A}$ are also normal and reductive,  the following corollary holds.
\begin{corollary}\label{corollary 3.2}
	Let $A\in\mathcal{A}$  with $A=U^*DU$, and let $\G$ be a countable set of vectors in $\ell^2(\mathbb{N})$.  Then, $\{A^tg\}_{g\in\G,t\in[0,L]}$ is complete in $\ell^2(\mathbb{N})$ if and only if $\{P_{j}(Ug)\}_{g\in\G}$ is complete in $E_j$.
\end{corollary}
The proof of Theorem \ref{theorem 3.1} below, also shows that, for normal reductive operators, completeness in $\HH$ is equivalent to completeness of the system $\{N_{\mu_j}^t\tilde{g}_j\}_{g\in\G,t\in[0,L]}$ in $(L^2(\mu_j))^{(j)}$ for every $1\leq j\leq \infty$. In other words, the completeness of $\{A^tg\}_{g\in\G,t\in[0,L]}$ is equivalent to the completeness of its projections onto the mutually orthogonal subspaces $U^*P_{j}U\HH$ of $\HH$. The following Theorem \ref{theorem 3.4} summarizes the discussion above.
\begin{theorem}\label{theorem 3.4}
	Let $A\in\mathcal{B}(\HH)$ be a normal reductive operator on the Hilbert space $\HH$, and let $\G$ be a countable system of vectors in $\HH$. Then, $\{A^tg\}_{g\in\G,t\in[0,L]}$ is complete in $\HH$ if only if the system $\{N^t_{\mu_j}\tilde{g}_j\}_{g\in\G,t\in[0,L]}$ is complete in $(L^2(\mu_j))^{(j)}$ for every $j$, $1\leq j\leq\infty$.
\end{theorem}

\subsection{Proofs}
We begin this section by stating and proving a lemma  used to prove Theorem \ref{theorem 3.1} as well as other results in later sections.\\

Let $A$ be a normal operator, $L$ be a positive number, $f\in\HH$, $\tilde{f}=Uf=(\tilde f_j)$, and $\tilde{g}=Ug=(\tilde g_j)$ (as in the notation section).  Define $F(t)$ by   
\[F(t)=\langle A^tg,f\rangle=\int_{\mathbb{C}}z^t\langle \tilde{g}(z),\tilde{f}(z)\rangle d\mu(z).\]
%Then one has
Then, the following lemma holds.

\begin{lemma}\label{lemma 3.6}
	$F(t)$ is an analytic function of $t$ in the domain $\Omega=\{t:\Re(t)>L/2\}$, where $\Re(t)$ stands for the real part of $t$.
\end{lemma}
\begin{proof}
	First, we aim to  
	prove that $F(t)$ is a continuous function in $\Omega$. Consider $t_0\in\Omega$. For $|z|\le M,$ where $M=\|A\|$, and for $t\in\Omega$ with $|t-t_0|<L/4$,
	one has
	\begin{eqnarray*}
		|z^t\langle\tilde{g}(z),\tilde{f}(z)\rangle|&=& |e^{t\ln(z)}||\langle \tilde{g}(z),\tilde{f}(z)\rangle|\\
		&\leq&e^{(|\ln (M)|+\pi)|t|}|\langle \tilde{g}(z),\tilde{f}(z)\rangle|\\
		&\leq&e^{(|\ln (M)|+\pi)(|t_0|+\frac L 4)}|\langle \tilde{g}(z),\tilde{f}(z)\rangle|.		\end{eqnarray*}

	Since the right hand side of the last inequality is an $L^2(\mu)$ function, we can use the dominated convergence theorem for $\Re(t)>L/2>0$, and get that for $t_0\in\Omega$,
	\begin{equation*}
	\lim_{t\rightarrow t_0}F(t)=\lim_{t\rightarrow t_0}\int_\CC z^t\langle \tilde{g}(z),\tilde{f}(z)\rangle d\mu(z)=\int_\CC\lim_{t\rightarrow t_0}z^t\langle\tilde{g}(z),\tilde{f}(z)\rangle d\mu(z)=F(t_0).
	\end{equation*}
	Therefore, $F(t)$ is a continuous function in $\Omega$. 
	
	Next we  show that for every closed piecewise $C^1$ curve $\gamma$ in $\Omega$,
	\[\oint_{\gamma}F(t)dt=0.\]
	
	For fixed $\gamma$,  there exists finite $M_1>0$ such that $L/2<|t|<M_1$.  Therefore, for $|z|\le M$, 
	\[|z^t\langle \tilde{g}(z),\tilde{f}(z)\rangle|\leq e^{\tilde{M}}|\langle \tilde{g}(z),\tilde{f}(z)\rangle|,\] with $\tilde{M}=M_1(|\ln M|+\pi)$.
	Then $$\oint_{\gamma}\int_\CC|z^t||\langle\tilde{g}(z),\tilde{f}(z)\rangle|d\mu(z)dt\leq e^{\tilde{M}}\|f\|_2\|g\|_2\cdot m_1(\gamma)<\infty, $$ where $m_1(\gamma)$ stands for the length of $\gamma.$
	
	By Fubini's theorem, 
	\begin{eqnarray*}
		\oint_{\gamma}\int_\CC z^t\langle\tilde{g}(z),\tilde{f}(z)\rangle d\mu(z)dt&=&\int_\CC\oint_{\gamma}z^t\langle\tilde{g}(z),\tilde{f}(z)\rangle dt d\mu(z)\\
		&=&\int_\CC\langle\tilde{g}(z),\tilde{f}(z)\rangle\oint_{\gamma}z^t dt d\mu(z)=0.
	\end{eqnarray*}
	where the last equality follows from the fact that for  $z\in\CC$, $h_z(t)=z^t$ is an analytic function of $t$ in $\Omega$  and hence $\oint_{\gamma}z^tdt=0$.  Then, by Morera's Theorem \cite[pp 208]{Rudinrc}, $F(t)$ is analytic on $\Omega$.
\end{proof}

\begin{proof}[\textbf{Proof of Theorem \ref{theorem 3.1}}]
	Since $\{A^tg\}_{g\in\G,t\in [0,L]}$ is complete in $\HH$,
	\[U\{A^tg:g\in\G,t\in [0,L]\}=\{(N^t_{\mu_j}\tilde{g}_j)_{j\in\mathbb{N}^*}:g\in\G,t\in [0,L]\}\] is complete in $\mathcal{W}=U\HH$. Hence, for every $1\leq j\leq\infty$, the system $\widetilde{\mathcal{S}}_j=\{N^t_{\mu_j}\tilde{g}_j\}_{g\in\G,t\in[0,L]}$ is complete in $(L^2(\mu_j))^{(j)}.$
	
	To finish the proof of the first statement of Theorem \ref{theorem 3.1} we use the following lemma, which is an adaptation of \cite[Lemma 1]{kriete} (\cite[Lemma 3.5]{ACCMP17}).
	
	\begin{lemma}\label{lemma 3.5}
		Let $\mathfrak{S}$ be a complete countable set of vectors in $(L^2(\mu_j))^{(j)}$, then for $\mu_j$-almost every $z$, $\{h(z):h\in\mathfrak{S}\}$ is complete in $\ell^2(\Omega_j)$.
	\end{lemma}
	Since $\HH$ is separable, there exists a countable set  $T=\{t_i\}_{i=1}^{\infty}\subseteq [0,L]$ with $t_1=0$ such that $\overline{span}\{A^{t}g\}_{g\in\G,t\in T}=\overline{span}\{A^{t}g\}_{g\in\G,t\in [0,L]}$.  Hence, the fact that $\widetilde{\mathcal{S}}_j=\{N^t_{\mu_j}\tilde{g}_j\}_{g\in\G,t\in [0,L]}$ is  complete in $(L^2(\mu_j))^{(j)}$ (together with Lemma \ref{lemma 3.5}) implies that $\{z^{t}\tilde{g}_j(z)\}_{g\in\G,t\in T}$ is complete in $\ell^2(\Omega_j)$ for each $j\in\mathbb{N}^*$. %In addition, since $z^{t_1}=z^0=1 $, %	 
	Let $f\in\HH$ and $F(t)=\langle A^tg,f \rangle=0$ for all $g\in\G,t\in [0,L]$.	 Since $F(t)=0$ for all $t\in[0,L]$, and $F$ is analytic for $t \in \Omega=\{t:\Re(t)>L/2\}$, it follows that $F(t)=0,$ for all $t\in\Omega$ (see \cite[Theorem 10.18]{Rudinrc}). Thus, $F(n)=0$ for all $n\in\mathbb{N}$, i.e., $\text{for all } n\in\mathbb{N}$,
	\begin{equation}
	\int_\CC z^n\langle\tilde{g}(z),\tilde{f}(z)\rangle d\mu(z)=	\int_{\mathbb{C}}z^n\left[\sum_{1\leq j\leq\infty}\chi_{\mathcal{E}_j}(z)\langle\tilde{g}_j(z),\tilde{f}_j(z)\rangle_{\ell^2(\Omega_j)} \right]d\mu(z)=0.
	\end{equation} 
	To finish the proof, we need the following proposition from \cite{wermer}.
	\begin{proposition}\label{proposition reductive}
		Let $A$ be a normal operator on the Hilbert space $\HH$ and let $\mu_j$ be
		the measures in the representation \eqref{representation of normal} of $A$. Let $\mu$ be as in \eqref{scalar spectral measure}. Then, A is reductive if and only if, for any two vectors $f, g \in\HH$,
		$$\int_{\mathbb{C}}z^n\left[\sum_{1\leq j\leq\infty}\chi_{\mathcal{E}_j}(z)\langle\tilde{g}_j(z),\tilde{f}_j(z)\rangle_{\ell^2(\Omega_j)} \right]d\mu(z)=0$$
		for every $n\geq 0$ implies $\mu_j$-a.e. $\langle \tilde{g}_j(z),\tilde{f}_j(z)\rangle_{\ell^2(\Omega_j)}=0$ for every $j\in\mathbb{N}^*$. 
	\end{proposition}
	Since $A$ is reductive,  it follows from Proposition \ref{proposition reductive} that $\langle \tilde{g}_j(z),\tilde{f}_j(z)\rangle_{\ell^2(\Omega_j)}=0$ for every $j\in\mathbb{N}^*$.
	Finally, since $\{\tilde{g}_j(z)\}_{g\in\G}$ is complete in $\ell^2(\Omega_j)$ for $\mu_j$-a.e. $z$,  we get that  $\tilde{f}_j(z)=0, \mu_j\text{-a.e. z} \text{ for every } j\in\mathbb{N}^*$. Thus, $\tilde{f}=0$ $\mu$-a.e. $z$, and hence  $f=0$. Therefore,    $\{A^tg\}_{g\in\G,t\in[0,L]}$ is complete in $\HH$.  
\end{proof}

\section{Bessel system}\label{bessel section}
The goal of this section is to study the conditions for which the system $\{A^tg\}_{g\in\G,t\in[0,L]}$ is  Bessel in $\HH$. There are two main theorems that correspond to the finite dimensional case and the infinite dimensional case, respectively.  The proofs  of the results are relegated to the last subsection. We begin with the following proposition  which is valid for both finite and infinite dimensional spaces.

\begin{proposition}\label{BesselEqu}
	Let $A \in \mathcal B(\HH)$ be normal, $\G\subset \HH$ be a countable set of vectors, and let $L$ be a positive finite number. If $\G$ is a Bessel system in $\HH$, then $\{A^tg\}_{g\in\G,t\in[0,L]}$ is a Bessel system in $\HH$.
\end{proposition}

The fact  that $\G$ is a Bessel system in $\HH$ implies that $\{A^tg\}_{g\in\G,t\in[0,L]}$ is Bessel in $\HH$ is not too surprising. However, the converse implication is not obvious. The next result characterizes the finite dimensional case.

\begin{theorem}[Besselness in  finite dimensional space]\label{FDBessel}
	Let	$A$ be  a normal operator on $\CC^d$ and $L$ be a positive finite number.  Let $M=Range(A^*)$ and $P_M\G=\{P_Mg\}_{g \in \G}$, where $P_M$ is the orthogonal projection on $M$. Then, $\{A^tg\}_{g\in\G,t\in[0,L]}$ is a Bessel system in $\CC^d$ if and only if $P_M\G$ is a Bessel system in $M$.
\end{theorem}

Under the appropriate restrictions on the spectrum $\sigma (A)$  of $A$, one can obtain a result similar to Theorem \ref {FDBessel} for the infinite dimensional case. However, if $0 \notin \sigma (A)$, the main result for the infinite dimensional Hilbert space is stated in the following theorem. 
\begin{theorem} [Besselness in infinite dimensions]  \label {MainBessel ID}
	Let $A\in\mathcal{B}(\HH)$ be an invertible normal  operator, and  let $\G$ be a countable system of vectors in $\HH$. Then, for any finite positive number  $L$, $\{A^tg\}_{g\in\G,t\in[0,L]}$ is a Bessel system in $\HH$
	if and only if $\G$ is a Bessel system  in $\HH$.
\end{theorem}

The condition that $A$ is invertible is necessary in Theorem \ref{MainBessel ID} as can be shown by the following example.

\begin{example}\label{example1}
	Let $\G=\{ne_n\}_{n=1}^{\infty}$ with $\{e_n\}_{n=1}^{\infty}$ being the standard basis of $\ell^2(\mathbb{N})$, $f\in\ell^2(\mathbb{N})$ with $f(n)=1/n$, and let $D$ be the diagonal infinite matrix with diagonal entries $D_{n,n}=e^{-{n^2}}$.  The operator $D$ is injective but  not an invertible operator on  $\ell^2(\mathbb{N})$. 
	
	Note that $$\sum_{g\in \G}|\langle f,g\rangle|^2=\infty.$$ Hence, $\G$ is not a Bessel system in $\ell^2(\mathbb{N})$.
	On the other hand,
	\begin{equation}
	\sum_{g\in\mathcal{G}}\int_{0}^{1}|\langle f,D^tg\rangle|dt= \sum_{n=1}^{\infty}\frac{1-e^{-2n^2}}{2}|f_n|^2\leq\|f\|^2/2.
	\end{equation} Thus $\{D^{t}g\}_{g\in\mathcal{G},t\in[0,1]}$ is  Bessel in $\ell^2(\mathbb{N})$.
\end{example}
\subsection {Proofs for Section \ref {bessel section}}
\begin{proof}[\textbf{Proof of Proposition \ref{BesselEqu}}]
	For all $ f\in\HH$,
	\begin{eqnarray}
	\sum\limits_{g\in\G}\int\limits_{0}^{L}|\la  f,A^tg\ra|^2 dt&=&	\sum\limits_{g\in\G}\int\limits_{0}^{L}|\la  A^{*t}f,g\ra|^2 dt\nonumber\\
	&=&\int\limits_{0}^{L}\sum\limits_{g\in\G}|\la A^{*t} f,g\ra|^2 dt
	\leq\int\limits_{0}^{L}C_\G\|A^{*t}f\|^2dt\nonumber\\ &\leq&\int\limits_{0}^{L}C_\G\|A\|^{2t}\|f\|^2dt=\begin{cases}
	C_\G\frac{\|A\|^{2L}-1}{\ln\|A\|^2}\|f\|^2, \|A\|\neq 1\\
	C_{\G}L\|f\|^2,\|A\|=1,\\
	\end{cases}\nonumber
	\end{eqnarray}
	where $C_{\G}$ is a Bessel constant of the Bessel system $\G$. 
	Therefore, $\{A^tg\}_{g\in\G,t\in[0,L]}$ is Bessel in $\HH$.
\end{proof}

In order to prove Theorem \ref {FDBessel}, we need the following lemma:
\begin{lemma} \label{FDIG}    
	Let $\G=\{g_{j}\}_{j\in J}\subset \CC^d$ where $J$ is a countable set. Then, $\G$  is a Bessel system if and only if $\sum_{j\in J}\|g_j\|^2<\infty$. 
\end{lemma}

\begin{proof}[\textbf{Proof of Lemma \ref {FDIG}}]
	$(\Longrightarrow)$Let $\{u_i\}_{i=1}^d$ be an orthonormal basis in $\CC^d$. If $\{g_j\}_{j\in J}$ is a Bessel system with Bessel constant $C$, then,  for  $i=1,\ldots,d$
	\[\sum_{j\in J}|\langle u_i,g_j\rangle|^2 \leq C.\]
	Since  $\|g_j\|^2=\sum_{i=1}^{d}|\langle u_i,g_j\rangle|^2$ for $j\in J$, one obtains
	$$\sum_{j\in J}\|g_j\|^2=\sum_{j\in J}\sum_{i=1}^d|\langle u_i,g_j\rangle|^2\leq Cd<\infty.$$ 
	
	$(\Longleftarrow)$ For any $f\in\HH$, one has
	\begin{eqnarray*}
		\sum_{j\in J}|\langle f,g_j\rangle|^2\leq \sum_{j\in J}\|f\|^2\|g_j\|^2=\|f\|^2(\sum_{j\in J}\|g_j\|^2).
	\end{eqnarray*}
	Therefore, $\{g_j\}_{j\in J}$ is Bessel in $\CC^d$.
\end{proof}	

\begin{proof}[\textbf{Proof of Theorem \ref{FDBessel}}]
	$(\Longleftarrow)$ Since $A$ is a normal operator on $\HH=\CC^d$,  it is clear that $A=\sum_{i\in I}\lambda_iP_i$ where $P_iP_j=0$  for $i\ne j$, $I=\{i:\lambda_i\neq 0\}$, and $(\sum_{i\in I} P_i)(\CC^d)=M$, where $M=Range(A^*)=Null^\perp (A)=Null^\perp (A^*)$.
	
	For $f\in\CC^d$,  one has
	\begin{eqnarray}
	\sum_{g\in G}\int_{0}^{L}|\langle f,A^tg\rangle|^2dt&=&\sum_{g\in G}\int_{0}^{L}\left|\langle A^{*t}f,g\rangle\right|^2dt
	=\sum_{g\in G}\int_{0}^{L}\left|\sum_{i\in I}\overline{\lambda_i}^t\langle P_if,P_ig\rangle\right|^2dt\nonumber\\
	&\le&\sum_{g\in G}\int_{0}^{L}\|A\|^{2t}\left(\sum_{i\in I}\lvert{\langle P_if,P_ig\rangle}\rvert\right)^2dt\nonumber\\
	&\leq&\int_{0}^{L}\|A\|^{2t}\sum_{g\in G}\left(\sum_{i\in I}\|P_if\|^2\right)\left(\sum_{i\in I}\|P_ig\|^2\right)dt\nonumber\\
	&\leq&\int_{0}^{L}\|A\|^{2t}\|P_Mf\|^2 \sum_{g\in G}\|P_Mg\|^2dt\leq C_1\cdot C_{P_M\G}\cdot \|f\|^2, \nonumber
	\end{eqnarray}	
	where $ C_1=\begin{cases}
	(\|A\|^{2L}-1)/\ln(\|A\|^2),\|A\|\neq 1\\
	L,\|A\|=1\\
	\end{cases}$ and $C_{P_M\G}=\sum_{g\in G}\|P_Mg\|^2$.
	
	In addition, one can use  Lemma \ref {FDIG} to conclude that $C_{P_M\G}=\sum_{g\in G}\|P_Mg\|^2< \infty$. Therefore, $\{A^tg\}_{g\in\G,t\in[0,L]}$ is  Bessel in $\CC^d$.\\
	($\Longrightarrow$)
	Since $A$ is normal, $A$ can be written as $A=\sum_{i\in I}\lambda_iP_i$, with $\rank(P_i)=1$ (in this representation, we allow $\lambda_i=\lambda_j$ for $i\ne j$) and $I=\{i:\lambda_i\neq 0\}$, $P_iP_j=0$   for $i\ne j$,  and $(\sum_{i\in I}P_i)(\CC^d)=M$. 	
	Specifically, by setting $f=u_i$, where $u_i$ is a unit vector in the one dimensional space $P_i(\CC^d)$ with $i\in I$, one has
	\begin{eqnarray*}
	\sum_{g\in\G}\int_{0}^{L}|\langle u_i,A^tg\rangle|^2dt&=&\sum_{g\in\G}\int_{0}^{L}|\langle u_i,\lambda_i^tP_ig\rangle|^2dt \\
	&=&\sum_{g\in\G}\int_{0}^{L}\lvert \lambda_i\rvert^{2t}\|P_ig\|^2dt \\
	&=&\begin{cases}
	L\sum_{g\in\G}\|P_ig\|^2,\quad |\lambda_i|=1\\
	\frac{|\lambda_i|^{2L}-1}{2\ln|\lambda_i|}\sum_{g\in\G}\|P_ig\|^2,\text{ otherwise}.
	\end{cases}.
	\end{eqnarray*}
	In addition, since by assumption $\{A^tg\}_{g\in\G,t\in[0,L]}$ is a Bessel system in $\CC^d$ with Bessel constant $C$, then
	$\sum_{g\in\G}\int_{0}^{L}|\langle u_i,A^tg\rangle|^2dt\leq C\|u_i\|^2=C$. Hence, for each $i$,
	\begin{equation*}
	\sum_{g\in\G}|\\|P_ig\|^2<\infty.
	\end{equation*}
	Therefore, summing over (the finitely many) $i\in I$ we obtain
	\begin{equation*}
	\sum_{g\in\G}|\\|P_Mg\|^2<\infty.
	\end{equation*}
	If $f\in M=Range(A^*)$, then
	\begin{eqnarray*}
	\sum_{g\in\G}|\langle f,g\rangle|^2&=&\sum_{g\in\G}\left\lvert \sum_{i\in I} \langle P_if,P_ig\rangle\right\rvert^2\nonumber\\
	&\le&\sum_{g\in\G}\left({\sum_{i\in I} \|P_if\|^2}\right)\left({\sum_{i\in I} \|P_ig\|^2}\right)\nonumber\\
	&=&
	\|f\|^2\sum_{g\in\G} \|P_Mg\|^2.
	\end{eqnarray*}
	Thus, $P_M\G$ is Bessel in $M$.
\end{proof}
Before  proving Theorem \ref{MainBessel ID}, we first state and prove the following  lemmas.
\begin{lemma}\label{BT}
	Let $A\in\mathcal{B}(\HH)$ be an invertible  operator in $\HH$, then a countable set $\G \subseteq \HH$ is a Bessel system in $\HH$ if and only if $\tilde{\G}=A\G$ is   a Bessel system  in $\HH$.
\end{lemma}
\begin{proof}[\textbf{Proof of Lemma \ref{BT}}]
	$(\Longrightarrow)$ 
	For all $ f\in\HH$,
	\begin{eqnarray*}
		\sum_{g\in\G}|\langle f,Ag\rangle|^2&=&\sum_{g\in\G}|\langle A^*f,g\rangle|^2\\
		&\leq&C\|A^*f\|_2^2\leq C\|A\|_2^2\|f\|_2^2,
	\end{eqnarray*}where $C$ is a Bessel constant of the Bessel system $\G$.
	Therefore, $A\G$ is a Bessel system in $\HH.$\\
	$(\Longleftarrow)$ For all $f\in\HH$,
	\begin{eqnarray*}
		\sum_{g\in\G}|\langle f,g\rangle|^2&=&\sum_{g\in\G}\left|\langle (A^*)^{-1}f,Ag\rangle\right|^2\\
		&\leq&C_1\|(A^*)^{-1}f\|_2^2\leq C_1\|A^{-1}\|_2^2\|f\|_2^2,
	\end{eqnarray*}
	where $C_1$ is a Bessel constant of the Bessel system $A\G$.
	Therefore, $\G$ is a Bessel system in $\HH$. 
\end{proof}

%\begin{lemma}\label {SpecInvNorm} 
%	If	$A$ is an invertible normal operator in $\HH$, then for $z\in\sigma(A)$,  $\|A^{-1}\|^{-1}\leq |z|\leq\|A\|$.
%\end{lemma}
%\begin{proof}[\textbf{Proof of Lemma \ref{SpecInvNorm}}] The assertions are inferred  from the following two inequalities \eqref{upperI}, \eqref{lowerI}. 
%	
%	For $|z|>\|A\|$ and non-zero $f\in\HH$,
%	\begin{equation}\label{upperI}
%	\|(A-zI)f\|\geq |z|\|f\|-\|Af\|\geq(|z|-\|A\|)\|f\|.
%	\end{equation} 
%	By the non-singularity of $A$, one has $Af\neq 0$ for all non-zero $f\in\HH$. 
%	
%	By the boundedness of $A^{-1}$, for $|z|<\|A^{-1}\|^{-1}$ and all  non-zero $f\in\HH$, \begin{eqnarray}\label{lowerI}
%	\|(A-zI)f\|&\geq&\|Af\|-|z|\|A^{-1}(Af)\|\nonumber\\
%	&\geq&(1-|z|\|A^{-1}\|)\|Af\|\nonumber\\
%	&\geq&(1-|z|\|A^{-1}\|)\|A^{-1}\|^{-1}\|f\|.
%	\end{eqnarray}
%\end{proof}
\begin{proof}[\textbf{Proof of Theorem \ref{MainBessel ID}}]
	$(\Longleftarrow)$ See  Proposition \ref{BesselEqu}.	\\
	$(\Longrightarrow)$
	Since $A$ is a normal operator in $\HH$, by the Spectral Theorem,  there exists a unitary operator $U$ such that $$UAU^{-1}=N_{\mu_{\infty}}^{(\infty)}\oplus N_{\mu_1}^{(1)}\oplus N_{\mu_2}^{(2)}\oplus\ldots$$ and $\mu$ is defined as by \eqref{scalar spectral measure}.  Therefore,  the task of proving that  $\G$ is a Bessel system in $\HH$ is equivalent to the task of showing that $U\G$ is a Bessel system in $\mathcal{W}=U\HH$.
	Let  $T:\mathcal{W}\rightarrow \mathcal{W}$ be the operator defined by:
	\begin {equation}\label {DefofT} T \tilde{f}(z):=\int_0^\ell z^tdt\tilde{f}(z),\text{ for all } \tilde{f}\in\mathcal{W} \text{ and }z\in\sigma(A) \text{ with }\ell=\min\{L,1/2\}.
\end{equation}   
The condition that  $\ell=\min\{L,1/2\}$ ensures that $T$ is an invertible operator as will be proved later.  

By Lemma \ref{BT},  $U\G$ is a Bessel system in $\mathcal{W}$ if and only if $T(U\G)$ is a Bessel system in $\mathcal{W}$  as long as  $T$ is a bounded  invertible normal operator. The  fact that $T$ is a bounded invertible operator is stated in the following lemma whose proof is postponed till after the completion of the proof of   this theorem.
\begin{lemma}\label{boundedT}
	$T$ is a bounded invertible operator in $\mathcal{W}$.
\end{lemma}
So, to finish the proof of  Theorem \ref {MainBessel ID}, it only remains to show that $T(U\G)$ is a Bessel system in $\mathcal{W}$ which we do next. 

Since $\{A^tg\}_{g\in\G,t\in[0,L]}$ is a Bessel system in $\HH$, and  $0<\ell\leq L$, one has 
that, for all $f\in\HH$,
\begin{equation*}\label{Besselell}
\sum_{g\in\G}\int_{0}^{\ell}|\langle f,A^tg\rangle|^2dt\leq C\|f\|^2.
\end{equation*} 
Thus, using H\"older's inequality, we get
\begin{equation} \label{fTg}\sum_{g\in\G}\left|\int_{0}^{\ell}\langle f,A^tg\rangle dt\right|^2\leq \ell\cdot \sum_{g\in\G}\int_{0}^{\ell}|\langle f,A^tg\rangle|^2dt\leq\ell C\|f\|^2. 
\end{equation}
In addition,
\begin{eqnarray} 
\sum_{g\in\G}\left|\int_{0}^{\ell}\langle f,A^tg\rangle dt\right|^2&=&\sum_{g\in\G}\left|\int_{0}^{\ell}\int_\CC \overline{z}^t\langle \tilde{f}(z),\tilde{g}(z)\rangle d\mu(z)dt\right|^2\nonumber\\
&=&\sum_{g\in\G}\left|\int_\CC  \int_{0}^{\ell}\overline{z}^t dt\langle\tilde{f}(z),\tilde{g}(z)\rangle d\mu(z)\right|^2\nonumber\\
&=&\sum_{g\in\G}|\langle \tilde{f},T\tilde{g}\rangle|^2\label{inequation 18}.
\end{eqnarray}
Together, \eqref{fTg} and \eqref{inequation 18} induce  the following inequality:
\begin{equation*}
\sum_{g\in\G}|\langle\tilde{f}, T\tilde{g}\rangle|^2\leq\ell C\|f\|^2=\ell C\|\tilde{f}\|^2,\text{ for all } f\in\HH.
\end{equation*}
This shows that $T(U\G)$ is a Bessel system in $\mathcal{W}$. 

In conclusion, by Lemma \ref{boundedT}, $T$ is bounded invertible. In addition, $T$ is normal. Hence, $U\G$ is a Bessel system in $\mathcal{W}$ by Lemma \ref{BT}. Consequently, $\G$ is a Bessel in $\HH$.
\end{proof}

\begin{proof}[\textbf{Proof of Lemma \ref{boundedT}}]
	\begin{eqnarray*}
		\|T\tilde{f}\|^2&=&\langle T\tilde{f},T\tilde{f}\rangle\\
		&=&\left\langle \int_0^\ell z^tdt\tilde{f}(z),\int_0^\ell z^{\tau}d\tau\tilde{f}(z)\right\rangle_{L^2(\sigma(A))}\\
		&=&\int_\CC \int_0^\ell \int_{0}^{\ell} z^t \overline{z}^{\tau}\langle \tilde{f}(z),\tilde{f}(z)\rangle dtd\tau d\mu(z)\\
		&=&\int_\CC |\phi(z)|^2\|\tilde{f}(z)\|^2d\mu(z),
	\end{eqnarray*} 
	where
	\begin{equation} \label {defPhi}
	\phi(z)=	\begin{cases}
	\ell,\quad z=1\\
	0,\quad z=0\\
	\frac{z^{\ell}-1}{\ln(z)},\text{ otherwise }
	\end{cases}.
	\end{equation}
	Let $m=\inf\{|\phi(z)|:z\in\sigma(A)\}$ and $M=\sup\{|\phi(z)|:z\in\sigma(A)\}$. As  shown below in claim \ref {lemma 4.11}, $m>0$ and $M<\infty$. Thus

	\begin{eqnarray*}
		\|T\tilde{f}\|^2&\leq&\int_\CC M^2\|\tilde{f}(z)\|^2d\mu(z)=M^2\|\tilde{f}\|^2,\\
		\|T\tilde{f}\|^2&\geq&\int_\CC m^2\|\tilde{f}(z)\|^2d\mu(z)=m^2\|\tilde{f}\|^2,\text{ for all } \tilde{f}\in\mathcal{W}.
	\end{eqnarray*}
	Since $T$ is normal, it follows  that $T$ is a bounded invertible operator (see \cite[Theorem 12.12]{rudinfa91}).
	We finish by proving the following fact that was used in the proof of this lemma.
	\begin{claim}\label{lemma 4.11} Let $\phi$ be the function defined in \eqref{defPhi}. Then
		$M=\sup\{|\phi(z)|:z\in\sigma(A)\}<\infty$,  and $m=\inf\{|\phi(z)|:z\in\sigma(A)\}>0$.
	\end{claim}
	\noindent {\bf Proof of Claim \ref{lemma 4.11}.}
	Since $A$ is a bounded invertible normal operator,  it follows that $\|A^{-1}\|^{-1}\leq|z|\leq\|A\|$ for $z\in\sigma(A)$. Let $S=\{z\in\mathbb{C}:\|A^{-1}\|^{-1}\leq|z|\leq\|A\|\}$. Since $\sigma(A)\subset S$,  $M\leq\sup\{|\phi(z)|:z\in S \}$ and $m\geq\{|\phi(z)|:z\in S\}$. Therefore, in order to prove Claim \ref{lemma 4.11}, it is sufficient to  show that
	$\sup\{|\phi(z)|:z\in S\}<\infty$, $\inf\{|\phi(z)|:z\in S\}>0$.
	
	To prove that $\sup\{|\phi(z)|:z\in S\}<\infty$, it is noteworthy  that 
	\begin{eqnarray*}
		|\phi(z)|&=&\left|\int_{0}^{\ell}z^{t}dt\right|\leq\int_{0}^{\ell}|z^t|dt	=\int_{0}^{\ell}|z|^tdt=\begin{cases}
			\ell,\quad \quad z\in S\text{ and } |z|=1\\
			\frac{|z|^{\ell}-1}{\ln|z|},z\in S\text{ and }|z|\neq 1.
		\end{cases}
	\end{eqnarray*}
	Let  $$\psi(x)=\begin{cases}
	\ell,x=1\\
	\frac{x^{\ell}-1}{\ln x},x\in\mathbb{R}^+\setminus\{1\},
	\end{cases}$$
	and note that (since  $\lim_{x\rightarrow 1}\frac{x^{\ell}-1}{\ln x}=\ell=\psi(1)$)   $\psi$ is  continuous at $x=1$. In addition, $\frac{x^{\ell}-1}{\ln x}$ is a continuous function on $\mathbb{R}^+\setminus\{1\}$. Hence, $\psi$ is continuous  on $\mathbb{R}^+$. Particularly, $\psi$ is continuous on the closed interval $[\|A^{-1}\|^{-1},\|A\|]$. Therefore,	  $$\sup\{|\phi(z)|:z\in S\}=\max_{x\in[\|A^{-1}\|^{-1},\|A\|]}\psi(x)<\infty.$$ 
	
	Finally, it remains to show that $\inf\{\left|\phi(z)\right|:z\in S\}>0$. First, we divide $S$ into two sets with $S_1=\{z\in S:\arg(z)\in[-\pi/2,\pi/2]\}$ and $S_2=S\setminus S_1$. Since $|\phi(z)|$ is a continuous function on $S_1$ and $S_1$ is compact, there exists $z_0\in S_1$ such that $|\phi(z_0)|=\inf\{|\phi(z)|:z\in S_1\}$. In addition, $|\phi(z)|$ has no root on $S_1$. Hence, $\inf\{|\phi(z)|:z\in S_1\}>0$.
	
	For $z\in S_2$, $\pi/2\leq|\arg(z)|\leq\pi$. Therefore, 
	\begin{eqnarray*}
		|z^{\ell}-1|&=&||z|^{\ell}e^{i\ell\arg(z)}-1|\\
		&\geq&|z|^{\ell}|\sin(\ell\arg(z))|\\
		&\geq&\min\{\|A^{-1}\|^{-\ell} \sin(\ell\pi),\|A^{-1}\|^{-\ell} \sin(\ell\pi/2)\}>0,
	\end{eqnarray*}
	where the last inequality follows from the fact that $0<\ell<1$ (in particular,  we  chose $\ell=\min\{L,1/2\}$ as in Definition \eqref {DefofT} for $T$).
	In addition, for $z \in S_2$, one has 
	\begin{eqnarray*}
		|\ln(z)|&\leq&|\ln(|z|)|+|\arg(z)|\leq\max\{|\ln(\|A^{-1}\|^{-1})|,|\ln(\|A\|)|\}+\pi<\infty.
	\end{eqnarray*}
	 Hence, $\inf\{|\phi(z)|: z\in S_2 \}>0$. 
	Combining the estimates on $S_1$ and $S_2$, we conclude that 
	$\inf\{|\phi(z)|:z\in S\}>0$. 
\end{proof}

\section{Frames generated by the action of bounded normal operators.}\label{frames}
In this section, we study some properties of a semi-continuous frame of the form $\{A^tg\}_{g\in\G,t\in[0,L]}$ generated by the continuous action of a normal operator $A\in \mathcal{B}(\HH)$ and  relate them to the properties of the discrete systems generated by  its  time discretization. We  also  show that, under the appropriate conditions, if $\{A^tg\}_{g\in\G,t\in[0,L_1]}$ is a semi-continuous frame for some positive number $L_1$, then $\{A^tg\}_{g\in\G,t\in[0,L]}$ a semi-continuous frame for all $0<L<\infty$. Before presenting the two main theorems, we first provide some necessary conditions for obtaining semi-continuous frames, and treat some special cases. The proofs are postponed to Subsection \ref {ProofFrame}. 

The following proposition (whose proof is obtained by direct calculation) provides a necessary condition to ensure the lower bound of the semi-continuous frame generated by $A\in \mathcal{B}(\HH)$.

\begin{proposition}\label{FrLrBd}
	Let $A\in\mathcal{B}(\HH)$ be an invertible normal operator, 
	$L$ be a finite positive number, and $\G\subseteq \HH$ be a countable set of vectors. If, for all $f\in\HH$,
	\begin{equation}\label{equationlowerdis}
	\sum_{g\in\G}|\langle f,g\rangle|^2\geq c\|f\|^2, 
	\end{equation}
	where $c$ is a positive constant,
	then there exists a finite positive constant $C$ such that 
	\begin{equation}\label{equationlowercont}
	\sum_{g\in\G}\int_{0}^{L}|\langle f,A^tg\rangle|^2dt\geq C\|f\|^2,  \text{ for all } f\in\HH .
	\end{equation}
\end{proposition}
The converse of  Proposition \ref{FrLrBd} is false, even in  finite dimensional space as shown in  Example \ref{FrLrBdEx}. For the special case that  $A$ is  equivalent to a diagonal operator on $\ell^2(\mathbb{N})$ we get: 
\begin{lemma}\label{DOPCase}
	Let $A\in\mathcal{A}$, where $\mathcal{A}$ is defined in \eqref{equationA*},  and let $\G\subseteq \ell^2(\mathbb{N})$ be a countable set of vectors. If $\{A^tg\}_{g\in\G,t\in[0,L]}$ satisfies \eqref {equationlowercont} in $\ell^2(\mathbb{N})$, then
	$$\sum_{g\in\G}\|g\|^2=\infty. $$
\end{lemma}
From Lemma \ref {DOPCase}, it follows that the cardinality of $\G$ must be infinite as stated in the following corollary.
\begin{corollary}
	If the  assumptions of Lemma \ref {DOPCase} hold  then $|\G|=+\infty.$  In particular, $|\G|=+\infty $ if $\{A^tg\}_{g\in\G,t\in[0,L]}$ is a frame for $\ell^2(\mathbb{N})$.
\end{corollary}

The discretization of continuous frames is a central question and has been studied extensively (see \cite {FR05, FS16}  and the references therein). In particular, Freeman  and Speegle have found necessary and sufficient conditions for the discretization of continuous frames \cite {FS16}.  In our situation, the systems $\{A^tg\}_{g\in\G,t\in[0,L]}$ can be viewed as  continuous frames and the theory in \cite {FS16} may be applied to conclude that the system can be discretized. However, because of the particular structure of the systems $\{A^tg\}_{g\in\G,t\in[0,L]}$, we can say more and obtain finer results for their discretization, as stated in the following theorem. 
%and The following theorem connects the semi-continuous frame $\{A^tg\}_{g\in\G,t\in[0,L]}$ with its time discretization. 

\begin{theorem}\label{ScToDscr}
	Let $A\in\mathcal{B}(\HH)$ be a normal operator on the Hilbert space $\mathcal{H}$ and let $\G$ be a Bessel system of vectors in $\HH$. If  $\{A^tg \}_{g\in\G,t\in[0,L]}$ is a semi-continuous frame for $\mathcal{H}$, then there exists $\delta>0$ such that for any finite set $T=\{t_i:i=1,\ldots,n\}$ with $0=t_1< t_2<\ldots<t_n<t_{n+1}=L$ and $|t_{i+1}-t_{i}|<\delta$,  the system $\{A^{t}g\}_{g\in \G,t\in T}$ is a frame for $\HH$.
	
	If, in addition, $A$ is  invertible, then $\{A^tg\}_{g\in\G,t\in[0,L]}$ is a semi-continuous frame for $\mathcal{H}$  if and only if there exists a finite set  $T=\{t_i:i=1,\ldots,n\}$ and $0= t_1< t_2<\ldots<t_n<L$,  such that $\{A^{t}g\}_{g\in \G, t\in T}$ is a frame for $\mathcal{H}.$
\end{theorem}

Example \ref {example3} shows that  the condition that $A$ is invertible is  necessary   for the second statement of  Theorem \ref{ScToDscr}.

The next theorem shows that, under some appropriate conditions, if $\{A^tg\}_{g\in\G,t\in[0,L_1]}$ is a semi-continuous frame for some finite positive number $L_1$, then $\{A^tg\}_{g\in\G,t\in[0,L]}$ is a semi-continuous frame for any finite positive number $L$.  
\begin{theorem}\label{SCFrSA}
	Let  $A\in\mathcal{B}(\HH)$ be an  invertible self-adjoint operator and $\G$ be a countable set in $\HH$. 
	Then, $\{A^tg\}_{g\in\G,t\in[0,1]}$ is a semi-continuous frame in $\HH$ if and only if $\{A^tg\}_{g\in\G,t\in[0,L]}$ is a semi-continuous frame in $\HH$ for all finite positive $L$.
\end{theorem}
We postulate the following conjecture:
\begin {conjecture}  Theorem \ref {SCFrSA} remains true if $A$ is a normal reductive operator.
\end {conjecture}

This first example shows that  the converse of  Proposition \ref{FrLrBd} is false.
\begin{example}\label{FrLrBdEx} 
	Let $A=\begin{bmatrix}
	\epsilon& 0\\
	0&1\\
	\end{bmatrix}$ with $0<\epsilon<1$ and $g=\begin{bmatrix}
	1\\
	1\\
	\end{bmatrix}$. 
	
	Note that for $L>0$, $$\G_1=\left\{g=\begin{bmatrix}
	1\\
	1\\
	\end{bmatrix}, A^{L/2}g=\begin{bmatrix}
	\epsilon^{L/2}\\
	1\\
	\end{bmatrix}\right\}$$ is complete in $\mathbb{R}^2$. In addition, $A$ is a bounded invertible normal operator in $\mathbb{R}^2$. Therefore, $\G_1$ is a frame in $\mathbb{R}^2$. By Theorem \ref{ScToDscr}, $\{A^tg\}_{t\in[0,L]}$ is a semi-continuous frame in $\mathbb{R}^2$. 	However, the lower bound of \eqref{equationlowerdis} does  not hold for $\G=\{g\}$. For example, let  $f=\begin{bmatrix} 
	-1\\
	1\\
	\end{bmatrix}$, then $ \langle f,g\rangle=0$.  
\end{example}
This next example shows that the condition that $A$ is invertible is required for the second statement of  Theorem \ref{ScToDscr}. 
\begin{example}\label{example3}
	Let $\G=\{e_j\}_{j=1}^{\infty}$ be the standard basis of $\ell^2(\mathbb{N})$. Because $\G$ is an orthonormal basis,  one has   $\G\subseteq \{A^{t}g\}_{g\in \G, t\in T}$, for any bounded operator $A$, and for any  time steps $T=\{t_i:i=1,\ldots,n\}$ with $0= t_1< t_2<\ldots<t_n<L$. Thus, $\G\subseteq \{A^{t}g\}_{g\in \G,t\in T}$ is a frame for $\ell^2(\mathbb{N}).$ 
	
	However, there exists a non-trivial bounded operator such that $\{A^te_j\}_{j\in\mathbb{N},t\in[0,L]}$ is not a semi-continuous frame. For example, if $D$ is a diagonal infinite matrix with diagonal entries $D_{j,j}=\frac{1}{j}$, then
	\begin{equation}
	\sum_{j=1}^{\infty}\int_{0}^{L}|\langle e_k,D^te_j\rangle|^2dt=\frac{1/k^{2L}-1}{\ln(1/k^{2})}.
	\end{equation}
	Since  
	$\lim\limits_{k\to\infty}\frac{1/k^{2L}-1}{\ln(1/k^{2})}=0, $
	it follows that $\{D^te_j\}_{j\in\mathbb{N},t\in[0,L]}$ is not a semi-continuous frame for $\ell^2(\mathbb{N})$.
\end{example}

Additionally, a number of examples are available to illustrate that $\{A^tg\}_{g\in\G,t\in[0,L]}$  is a semi-continuous frame for $\HH$ does not require $\G$ to be a frame  or even complete in $\HH$.  In fact, this is precisely why space-time sampling trade-off is feasible. The next two examples are toy examples to show this fact.
\begin{example} [$\G$ is not a frame for $\HH$]
	Let $\{e_n\}_{n=1}^{\infty}$ be the standard basis of $\ell^2(\mathbb{N})$ and $\G=\{g_n=e_n+e_{n+1}:n\in\mathbb{N}\}$, and let $D$ be a diagonal operator with 
	$D_{n,n}=	\begin{cases}
	1, n \text{ is odd}\\
	3, n\text{ is even}
	\end{cases}.$  \\
	
	It can be shown that $\G$ is complete  but that $\G$ is neither a basis nor a frame for $\ell^2(\N)$  \cite{C08}.  However, for all $f\in\ell^2(\mathbb{N})$, after a somewhat tedious computation, one gets
	\[\frac{1}{2}\|f\|^2\leq\sum_{n=1}^{\infty}\int_{0}^{1}|\langle f,D^tg_n\rangle|^2dt\leq\frac{16}{\ln(3)}\|f\|^2, \] so that $\{D^tg_n\}_{n\in\N,t\in[0,1]}$ is a semi-continuous frame for $\ell^2(\N)$.
\end{example}
\begin{example} [$\G$ is not complete in  $\HH$]
	Let $\{e_n\}_{n=1}^{\infty}$ be the standard basis of $\ell^2(\mathbb{N})$ and $\G=\{g_n=e_n+2e_{n+1}:n\in\mathbb{N}\}$.   The set $\G$ is not complete in $\ell^2(\mathbb{N})$. For example $f=(f_k)$ with $f_k=(-1)^k\frac 1 {2^k}$ is orthogonal to $\overline {\text {span } }\G$. Thus,  $\G$ is not a frame in $\ell^2(\N)$.  Let $D$  be the diagonal operator with 
	$$D_{n,n}=	\begin{cases}
	9, \quad \quad n=1\\
	1-\frac{1}{n},  n\geq 2
	\end{cases}.$$  
	A lengthy computation yields
	\begin{eqnarray*}
		\frac{1}{4}\|f\|^2\leq\sum_{n=1}^{\infty}|\langle f,g_n\rangle|^2+\sum_{n=1}^{\infty}|\langle f,Dg_n\rangle|^2\leq 164\|f\|^2.
	\end{eqnarray*}
	This implies that $\{D^tg\}_{g\in\G,t\in\{0,1\}}$ is a frame in $\ell^2(\mathbb{N})$. In addition, since $D$ is a self-adjoint  invertible  operator, Theorem \ref{ScToDscr} implies that
	$\{D^{t}g_n\}_{n\in\mathbb{N},t\in[0,2]}$ is a semi-continuous frame of $\ell^2(\mathbb{N})$.
\end{example}

\subsection {Proofs of Section \ref {frames}}
\label {ProofFrame}
\begin{proof}[\textbf{Proof of Lemma \ref {DOPCase}}]
	One can always assume that $A=\sum\limits_{i=1}^\infty\lambda_iP_i$ with $\rank (P_i)=1$, $P_iP_j=0$ and $\sum_iP_i=Id_{\ell^2(\N)}$ as long as  $\lambda_i=\lambda_j$ for $i\ne j$ in the representation of $A$ is allowed. 
	Let $e_i$ be a vector such that $\|e_i\|=1$ and $span \{e_i\}=P_i(\HH)$. Then 
	\begin{eqnarray*}
		\sum_{g\in\G}\int_{0}^{L}|\langle e_i,A^tg\rangle|_2^2dt&=&\sum_{g\in\G}\int_{0}^{L}|\lambda_i|^{2t}\vert\langle e_i,P_i(g)\rangle \vert^2 dt.		
	\end{eqnarray*}		
	Since  $\{A^tg\}_{g\in\G,t\in[0,L]}$  satisfies \eqref {equationlowercont}, we have that $\lambda_i\neq 0$ for all $i \in \N$.
	Moreover,  if $\sum_{g\in\G}\|g\|^2_2=\sum_{i\in\mathbb{N}}\sum_{g\in\G}\|P_ig\|^2<\infty$, then $\lim\limits_{i\to\infty}\sum_{g\in\G}\|P_ig\|^2=0.$ 
	In addition, since $ \frac{\|A\|^{2L}-1}{2\ln(\|A\|)}\geq\frac{|\lambda_i|^{2L}-1}{2\ln(|\lambda_i|)}>0,$ we get that  $ \lim\limits_{i\to \infty}\frac{|\lambda_i|^{2L}-1}{2\ln(|\lambda_i|)}\sum_{g\in\G}\|P_ig\|^2= 0. $ 
	This contradicts \eqref {equationlowercont}. Hence, $\sum_{g\in\G}\|g\|^2=\infty$. 
\end{proof}

\begin{proof}[\textbf{Proof of theorem \ref{ScToDscr}}]
	From the assumption that $\G$ is a Bessel sequence in $\HH$,  there exists $K>0$ such that $\sum_{g\in\G}|\langle f,g\rangle|^2\leq K\|f\|^2,$ for all $f\in\HH$. Since $A$ is a bounded normal operator, for any  $0\leq t<\infty$, one has
	\begin{equation}\label{AtgBessel}
	\sum_{g\in\G}|\langle f,A^{t}g\rangle|^2=\sum_{g\in\G}|\langle A^{*t}f,g\rangle|^2\leq K\|A^{*t}f\|^2\leq K\|A\|^{2 t}\|f\|^2.	\end{equation}
	Summing the  inequalities \eqref {AtgBessel}  over  $t\in T=\{t_i:i=1,\dots, n\}$, it immediately follows that $\{A^{t}g\}_{g\in\G,t\in T}$ is a Bessel sequence in $\HH$.
	
	Using \eqref {AtgBessel}, it follows that
	\begin{equation}\label{equation 25}
	\sum_{g\in\G}\int_{0}^{L}|\langle f,A^tg\rangle|^2dt\leq K\int_{0}^{L}\|A\|^{2t}dt\|f\|^2.
	\end{equation} 
	Inequality \eqref{equation 25} implies that
	for any $\epsilon>0$, there exists an $l$ with  $L/2>l>0$, such that 
	\begin {equation}\label {epsIflSmall}
	\sum_{g\in\G}\int_{0}^{l}|\langle f,A^tg\rangle|^2dt<\epsilon\|f\|^2.
\end{equation}

Next, the goal is to find $\delta>0$ such that for any finite set $T=\{t_i:i=1,\ldots,n\}$ with $0=t_1< t_2<\ldots<t_n<t_{n+1}=L$ and $|t_{i+1}-t_{i}|<\delta$,  the system $\{A^{t}g\}_{g\in \G,t\in T}$ is a frame for $\HH$, as long as  $\{A^tg \}_{g\in\G,t\in[0,L]}$ is a semi-continuous frame for $\mathcal{H}$, i.e.,  
\begin{equation}\label{equation 24}
c\|f\|^2\leq\sum_{g\in\G}\int_{0}^{L}|\langle f,A^tg\rangle |^2dt\leq C\|f\|^2, \quad \text {for all } f \in \HH,
\end{equation} 
for some $c, C>0$.

To finish the proof,  we use the following lemma.

\begin{lemma}\label{ContApower}
	Let $A\in \mathcal{B}(\HH)$ be a normal operator and $\ell,L$ be positive numbers with $0<\ell<L$. Then for any $\epsilon>0$, there exists $\delta>0$ such that whenever $s_1,s_2\in[\ell,L]$ with $|s_1-s_2|<\delta$, we have $\|A^{s_1}-A^{s_2}\|<\epsilon$.
\end{lemma}
\begin{proof}[\textbf{Proof of Lemma \ref{ContApower}}] 
	For $s_1,s_2\in[\ell,L]$,
	\begin{eqnarray*}
		|z^{s_1}-z^{s_2}|^2&=&|z|^{2s_1}-2|z|^{s_1}|z|^{s_2}\cos((s_1-s_2)arg(z))+|z|^{2s_2}\\
		&=&||z|^{s_1}-|z|^{s_2}|^2+2|z|^{s_1}|z|^{s_2}(1-\cos((s_1-s_2)arg(z))).
	\end{eqnarray*}
	For all $z\in\sigma(A)$, one has $0\leq|z|\leq\|A\|$.  Thus $|z|^{s}$ is uniformly bounded for all $s\in[\ell,L].$ In addition, the function $(t,r)\mapsto r^t$ is a continuous function on the compact set $ [ \ell,L]\times[0,\|A\|]$ and the function $t\mapsto\cos(t\cdot arg(z))$ is equicontinuous at $t=0$ for $arg(z)\in[-\pi,\pi)$. The lemma then follows from the spectral theorem (i.e., Theorem \ref {spectral theorem}).
	
\end{proof}
By Lemma \ref{ContApower}, there exists $\delta$ with $l/2>\delta>0$ such that  whenever $|s_1-s_2|<2\cdot\delta$ for $s_1,s_2\in[l/2,L]$, then $\|A^{s_1}-A^{s_2}\|<\epsilon$. Assume  that the set $T=\{t_i:i=1,\ldots,n\}$ satisfies $0=t_1< t_2<\ldots<t_n<t_{n+1}=L$ and $|t_{i+1}-t_{i}|<\delta$. Set $m=\min\{i:t_i>l/2\}$. Note that $l/2>\delta>0$. Therefore $t_m<l$. Then, using \eqref {epsIflSmall},  the difference 

%	By Lemma \ref{ContApower}, there exists $N\in\mathbb{N}$ such that  whenever $|s_1-s_2|<2\cdot\frac{L-l}{N}$ for $s_1,s_2\in[l,L]$, then $\|A^{s_1}-A^{s_2}\|<\epsilon$. Set $t_i=l+\frac{(i-1)(L-l)}{N}$. Then, using \eqref {epsIflSmall},  the difference 
\begin {equation}\label {DiffContDisc}
\Delta=\left|{\sum_{g\in\G}\int_{0}^{L}|\langle f,A^tg\rangle|^2dt-\sum_{g\in\G}\sum_{i=m}^{n}\int_{t_i}^{t_{i+1}}\ |\langle f,A^{t_i}g\rangle|^2}dt\right|,
\end {equation}
can be estimated  as follows.	
\begin{eqnarray*}
	\Delta
	&=&\left|{\sum_{g\in\G}\int_{0}^{L}|\langle f,A^tg\rangle|^2dt-\sum_{g\in\G}\sum_{i=m}^{n}\int_{t_i}^{t_{i+1}}\ |\langle f,A^{t_i}g\rangle|^2}dt\right| \\
	&\leq&\left(\sum_{g\in\G}\int_{0}^{t_m}|\langle f,A^tg\rangle|^2dt\right)+\sum_{i=m}^{n}\int_{t_i}^{t_{i+1}}\sum_{g\in\G}\vert {|\langle f,A^tg\rangle|^2-|\langle f,A^{t_i}g\rangle|^2}  \vert dt\\
	&=&\left(\int_{0}^{t_m}\sum_{g\in\G}|\langle f,A^tg\rangle|^2dt\right)+\sum_{i=m}^{n}\int_{t_i}^{t_{i+1}}\sum_{g\in\G}(|\langle f,A^tg\rangle|+|\langle f,A^{t_i}g\rangle|)(\vert |\langle f,A^tg\rangle|-|\langle f,A^{t_i}g\rangle| \vert )dt\\
	&\leq&\epsilon \|f\|^2+\sum_{i=m}^{n}\int_{t_i}^{t_{i+1}}\sum_{g\in\G}(|\langle A^{*t}f,g\rangle|+|\langle A^{*t_i}f,g\rangle|)(|\langle A^{*t}f-A^{*t_i}f,g\rangle|)dt\\
	&\leq&\epsilon\|f\|^2+\sum_{i=m}^{n}\int_{t_i}^{t_{i+1}}\left(\sum_{g\in\G}(|\langle A^{*t}f,g\rangle|+|\langle A^{*t_i}f,g\rangle|)^2\right)^{1/2}\left(\sum_{g\in\G}(|\langle A^{*t}f-A^{*t_i}f,g\rangle|)^2\right)^{1/2}dt\\
	&\leq&\epsilon\|f\|^2+\sum_{i=m}^{n}\int_{t_i}^{t_{i+1}}\left(2K(\|A^{*t}f\|^2+\|A^{*t_i}f\|^2)\right)^{1/2}(K\|A^{*t}f-A^{*t_i}f\|^2)^{1/2}dt\\
	&\leq&(\epsilon+2C_1 KL\epsilon)\|f\|^2,\text{ where }C_1=\max\{1,\|A\|^{L}\}.
\end{eqnarray*}
Using \eqref {DiffContDisc} and  choosing $\epsilon$ so small that  $(1+2C_1 KL)\epsilon<c/2$, we find $\delta$ such that 
$$\delta\sum_{g\in\G}\sum_{i=m}^{n}|\langle f,A^{t_i}g\rangle|^2\geq c\|f\|^2-c/2\|f\|^2=c/2\|f\|^2.$$ Therefore,  for any finite set $T=\{t_i:i=1,\ldots,n\}$ with $0=t_1< t_2<\ldots<t_n<t_{n+1}=L$ and $|t_{i+1}-t_{i}|<\delta$, the system  $\{A^{t}g\}_{g\in\G,t\in T}$ is a frame in $\HH$.

To prove the second statement,  
it is sufficient  to prove  that  $\{A^tg\}_{g\in\G,t\in[0,L]}$ is a semi-continuous frame under the assumption that $\{A^{t}g\}_{g\in\G,t\in T}$ is a frame in $\HH$ and  $A$ is an invertible normal operator. We already know by Theorem \ref{MainBessel ID} that $\{A^tg\}_{g\in\G,t\in[0,L]}$ is Bessel since $\G$ is Bessel by assumption. Let $T=\{t_i:i=1,\ldots,n\}$ with  $0=t_1<t_2<\ldots<t_n<L$ be such that $\{A^{t}g\}_{g\in\G,t\in T}$ is a frame for $\HH$ with frame constants $c,C$  i.e., for all $ f\in\HH$, 
\[c\|f\|^2\leq\sum_{g\in\G}\sum_{i=1}^{n}|\langle f,A^{t_i}g\rangle|\leq C\|f\|^2. \]  
Let $m=\min\{t_{i+1}-t_{i},1\leq i\leq n\}$ with $t_{n+1}=L$. Then, 
\begin{eqnarray*}
	\sum_{g\in\G}\int_{0}^{L}|\langle f,A^tg\rangle|^2dt&=&\sum_{g\in\G}\sum_{i=1}^{n}\int_{t_i}^{t_{i+1}}|\langle f,A^tg\rangle|^2dt\\
	&=&\sum_{g\in\G}\sum_{i=1}^{n}\int_{0}^{t_{i+1}-t_{i}}|\langle (A^{*t}f,A^{t_i}g\rangle|^2dt\\
	&\geq&\sum_{g\in\G}\sum_{i=1}^{n}\int_{0}^{m}|\langle A^{*t}f,A^{t_i}g\rangle|^2dt\\
	&\geq&\int_{0}^{m}c\|A^{*t}f\|^2_2dt.
\end{eqnarray*}
Since $A$ is an invertible bounded normal operator, we have % by Lemma \ref{SpecInvNorm},
\begin{eqnarray*}
	\int_{0}^{m}c\|A^{*t}f\|^2_2dt&\geq&c\cdot \frac{1-\|A^{-1}\|^{-2m}}{2\ln(\|A^{-1}\|)}\|f\|^2.
\end{eqnarray*}
This concludes the proof that  $\{A^tg\}_{g\in\G,t\in[0,L]}$ is a semi-continuous frame for $\HH$. 
\end{proof}
To prove Theorem \ref{SCFrSA}, the following  three lemmas, i.e., Lemmas \ref{SeriesConv}, \ref{AnalPosA} and \ref{ReImBessel} are needed.
\begin{lemma}\label{SeriesConv}
	Let  $\G$ be a countable Bessel sequence in $\HH$ and let $A\in\mathcal{B}(\HH)$ be a normal operator. Let  $L$ be any positive real number,   $\Omega_L=\{z:\Re(z)>L>0\}$, and let $\{g_i\}_{i\in I}$ be any indexing of $\G$.  Then,  for fixed $f\in\HH$,
	the partial sums  $\sum\limits_{i=1}^n|\langle A^zg_i,f\rangle|^2$ converge uniformly on any compact subset of $ \Omega_L$.
\end{lemma}
\begin{proof}[\textbf{Proof of Lemma \ref{SeriesConv}}]
	Let $\overline {D_r}$ denote the closed disk of radius $r$. Then using the fact that $\G$ is Bessel with Bessel constant $ C_\G$, for $z \in  \overline {D_r}\cap \overline {\Omega_L}$, one gets, 
	
	$$\sum\limits_{i=1}^n|\langle A^zg_i,f\rangle|^2=\sum\limits_{i=1}^n|\langle f,A^zg_i\rangle|^2=\sum\limits_{i=1}^n |\langle (A^{z})^*f,g\rangle|^2\le C_\G \cdot e^{2\pi r}\cdot\|A\|^{2r}\|f\|^2,$$
	from which the lemma follows.
\end{proof}
\begin{lemma}\label{AnalPosA}
	Let  $\G$ be a countable Bessel sequence in $\HH$ and let $A\in\mathcal{B}(\HH)$ be a  normal  operator. Let $L$ be any positive real number and  let $\Omega_L=\{z:\Re(z)>L>0\}$. Then, for  fixed $f\in\HH$,
	$$F(z)=\sum_{g\in\G}(\langle A^zg,f\rangle)^2,$$ is an analytic function of $z$ in $\Omega_L$.
\end{lemma}

\begin{proof}[\textbf{Proof of Lemma \ref{AnalPosA}}]
	Since $A$ is a  normal  operator on $\HH$, by Lemma \ref{lemma 3.6}, $\left(\langle A^zg,f\rangle\right)^2$ is  analytic in $\Omega_L$.
	Since $\left|\sum_{g\in\G}(\langle A^zg,f\rangle)^2\right|\leq\sum_{g\in\G}|\langle A^zg,f\rangle|^2$, by Lemma \ref{SeriesConv}, the series $\sum_{g\in\G}(\langle A^zg,f\rangle)^2$ converges absolutely and uniformly on any compact subset of $\Omega_{L}$, and the partial sums of $\sum_{g\in\G}(\langle A^zg,f\rangle)^2$ are analytic in $\Omega_{L}$ and converge uniformly on any compact subset of $\Omega_{L}$.  It follows that the series $\sum_{g\in\G}(\langle A^zg,f\rangle)^2$   is an analytic function of $z$ in $\Omega_L$ \cite[Theorem 10.28]{Rudinrc}.
\end{proof}
Let $A\in\mathcal{B}(\HH)$ be a normal operator, by the spectral theorem, there exists a unitary operator $U$ such that
$$UAU^{-1}=N_{\mu_{\infty}}^{(\infty)}\oplus N_{\mu_1}^{(1)}\oplus N_{\mu_2}^{(2)}\oplus\ldots.$$

For every $f\in\HH$, we define $\tilde{f}=Uf\in U\HH$. Note that $\tilde{f}:\sigma(A)\rightarrow \ell^2(\Omega_\infty)\oplus\ell^2(\Omega_1)\oplus\ell^2(\Omega_2)\oplus\ldots$ is a function and hence it makes sense to talk about its real and imaginary parts. Set $f^{\Re}=U^{-1}\Re(\tilde{f})$ and $f^{\Im}=U^{-1}\Im(\tilde{f})$.

\begin{lemma}\label{ReImBessel}
	If $\G$ is a Bessel sequence in $\HH$, then,  $\{g^{\Re}\}_{g\in\G}$ and $\{g^{\Im}\}_{g\in\G}$ are also Bessel sequences in $\HH$ for any given normal operator $A\in\mathcal{H}$. 
\end{lemma}

\begin{proof}[\textbf{Proof of Lemma \ref{ReImBessel}}]
	Consider the subspace $S\subseteq \HH$ defined by  $S=\{f\in\HH:Uf \text{ is real valued}\}$.
	Then,  for  $f\in S$, using the following identity  
	
	$$\sum_{g\in\G}|\langle f,g\rangle|^2=\sum_{g\in\G}|\langle \tilde{f},\tilde{g}\rangle|=\sum_{g\in\G}|\langle \tilde{f},\Re(\tilde{g})\rangle|^2+|\langle \tilde{f},\Im (\tilde{g})\rangle|^2=\sum_{g\in\G}|\langle f,g^{\Re}\rangle|^2+|\langle f, g^{\Im}\rangle|^2,  $$ it follows that $\{ g^{\Re}\}_{g\in\G}$ and $\{g^{\Im}\}_{g\in\G}$ are Bessel sequences in $S$. 	 
	For general $f\in\HH$, we have $f^\Re\in S$, $f^{\Im}\in S$, and
	$$ \sum_{g\in\G}|\langle f, g^\Re\rangle|^2= \sum_{g\in\G}|\langle f^\Re, g^\Re\rangle|^2+ \sum_{g\in\G}|\langle f^\Im, g^\Re\rangle|^2,$$ 
	$$ \sum_{g\in\G}|\langle f,g^\Im \rangle|^2= \sum_{g\in\G}|\langle f^\Re, g^\Im\rangle|^2+ \sum_{g\in\G}|\langle f^\Im, g^\Im\rangle|^2.$$ 
	It follows that	 $\{ g^\Re\}_{g \in\G}$ and $\{g^\Im\}_{g\in\G}$ are Bessel sequences for $\HH$.
	
\end{proof}

\begin{proof}[\textbf{Proof of Theorem \ref{SCFrSA}}]
	Assume that $\{A^tg\}_{g\in\G,t\in[0,1]}$ is a semi-continuous frame in $\HH$ with frame bounds $c$, $C$.
	By Theorem \ref{ScToDscr}, there exists a finite set $T$ such that $\{A^{t}g\}_{g\in\G,t\in T}$ is a frame for $\HH$. Therefore, for $L\geq 1$, $\{A^tg\}_{g\in\G, t\in[0,L]}$ is also a semi-continuous frame.
	
	To prove that $\{A^tg\}_{g\in\G, t\in[0,L]}$ is a semi-continuous frame for $L<1$, we note that the inequality
	$$\sum_{g\in\G}\int_{0}^{L}|\langle f,A^tg\rangle|^2dt\leq\sum_{g\in\G}\int_{0}^{1}|\langle f,A^tg\rangle|^2dt\leq C\|f\|_2^2$$ implies that 
	$\{A^tg\}_{g\in\G,t\in[0,L]}$ is a  Bessel system in $\HH$. Moreover, $A$ is an invertible bounded self-adjoint operator. Therefore, by Theorem \ref{MainBessel ID}, $\G$ is Bessel in $\HH$ with Bessel constant $C_{\G}.$
	
	Suppose that $\{A^tg\}_{g\in\G, t\in[0,L]}$ is not a frame. Then, there exists a sequence $\{f_n\}$ with $\|f_n\|=1$ such that $\sum_{g\in\G}\int_{0}^{L}|\langle f_n, A^tg\rangle|^2dt\rightarrow 0$. It follows that $\sum_{g\in\G}|\langle f_n,A^tg\rangle|^2\rightarrow 0$ in measure. Thus, there exists a subsequence $\{f_{n_k}\}$ of $\{f_n\}$ such that $\sum_{g\in\G}|\langle f_{n_k},A^tg\rangle|^2\rightarrow 0,$ for a.e. $t\in[0,L]$. 	By passing to a subsequence,  assume that $\sum_{g\in\G}|\langle f_n,A^tg\rangle|^2\rightarrow 0,$ for a.e. $t\in[0,L]$.
	
	To finish the proof, we next  prove  that there exists a subsequence $\{f_{n_k}\}$ of $\{f_{n}\}$ such that $$\sum_{g\in\G}\int_{0}^{1}|\langle f_{n_k},A^tg\rangle|^2dt\rightarrow 0.$$

	Since $A$ is a self-adjoint operator, by the spectral theorem, there exists a unitary operator $U$ such that $A$ can be represented as \eqref{representation of normal} and  $\sigma(A)\subseteq\mathbb{R}$. In addition, $A$ is  invertible.  %\textcolor{red}{By Lemma \ref{SpecInvNorm}, 
	Then there exist $m,M>0$ such that $m\leq|z|\leq M$ for all $z\in\sigma(A)$. Set $\tilde{f}=Uf$ and $\tilde{g}=Ug$. 	
	
	{\bf Case 1.} {\em  $A$ is a positive self-adjoint operator, and $\{Ug\}_{g\in\G}$ and $\{Uf_{n}\}$  are real-valued, i.e., $Ug=\Re(\tilde{g}) \text{ for all }g\in\G$ and $ Uf_n=\Re(\tilde f_n)$}: In this case, one has   $ |\langle f_{n},A^tg\rangle|^2=(\langle A^tg,f_n\rangle)^2, \text{ for all } t\in\mathbb{R}^+.$ Therefore 
	$$\sum_{g\in\G}|\langle f_{n},A^tg\rangle|^2=\sum_{g\in\G}(\langle A^tg,f_n\rangle)^2, \text{ for all } t\in\mathbb{R}^+.$$
	Moreover, since $\G$ is Bessel, by Lemma \ref{AnalPosA}, the functions $F_n(t)=\sum_{g\in\G}(\langle A^tg,f_n\rangle)^2$ are analytic for $ t \in \Omega_{L/4}\cap D_r\subseteq \CC$ and  satisfy 
	\begin{eqnarray*}
		|F_n(t)|=\left|\sum_{g\in\G}(\langle A^tg,f_n\rangle)^2\right|&\leq& \sum_{g\in\G}\left|\langle g,(A^{t})^*f_n\rangle\right|^2\leq C_{\G}\|A\|^{2r},  \text{ for } t \in \Omega_{L/4}\cap D_r.
	\end{eqnarray*}
	Thus, by Montel's theorem, there exists a subsequence $\{F_{n_k}\}$ of $\{F_n\}$ such that $\{F_{n_k}\}$  converge to an analytic function $F$ on $ \Omega_{L/4}\cap D_r$. Let $D_r \subset \CC$ be a disk of radius $r$  containing $[L/2,1]$.  Since $F_n$ are analytic and $F_n(t)\rightarrow 0,\text{ for all } t\in[L/2,L]$, it follows that $F(t)=0$, for all $t\in[L/2,L]$. Moreover, since  $F$ is analytic, we conclude that $F(t)=0$ for all $ t \in \Omega_{L/4}\cap D_r$, and hence also on $[L/2,1]$, i.e., $\lim_{n_k\rightarrow\infty}F_{n_k}(t)=0$ for all $t\in[L/2,1]$. Thus, 
	\begin{eqnarray*}
		&&\sum_{g\in\G}\int_{0}^{1}|\langle f_{n_k},A^tg\rangle|^2dt\\
		&=&\sum_{g\in\G}\int_{0}^{L/2}|\langle f_{n_k},A^tg\rangle|^2dt+\sum_{g\in\G}\int_{L/2}^{1}|\langle f_{n_k},A^tg\rangle|^2dt.
	\end{eqnarray*} Taking limits as $n_k$ tends to infinity, one sees that $\lim\limits_{n_k\rightarrow\infty}\sum_{g\in\G}\int_{0}^{1}|\langle f_{n_k},A^tg\rangle|^2dt=0$. This contradicts the assumption that $\{A^tg\}_{g\in\G,t\in[0,1]}$ is a semi-continuous frame. Therefore, $\{A^tg\}_{g\in\G, t\in[0,L]}$ is a semi-continuous frame.\\
	
	{\bf Case 2.} {\em The general case}:

	Let $\tilde {f}_n=\Re(\tilde{f}_{n}) +i\Im(\tilde{f}_{n}) $ and $\tilde {g}=\Re(\tilde{g}) +i\Im(\tilde{g}) $. Define $f_n^\Re=U^{-1}\Re(\tilde{f}_{n})$,  $f_n^\Im=U^{-1}\Im(\tilde{f}_{n})$, $g^\Re=U^{-1}\Re(\tilde{g})$, and $g^\Im=U^{-1}\Im(\tilde{g})$. Define $A_+^t$ and $A_{-}^t$ as 
	\begin {align*}
	\langle A_+^tg,f\rangle&=\int_{z\in\sigma(A),z>0}z^t\langle \tilde{g},\tilde{f}\rangle d\mu(z),\\ 
	\langle A_{-}^tg,f\rangle&=\int_{z\in\sigma(A),z<0}(-z)^t\langle \tilde{g},\tilde{f}\rangle d\mu(z).
\end{align*}
Then $A_{-}$ and $A_+$ are positive operators, and $\langle A^tg,f\rangle= \langle A_{+}^tg,f\rangle+e^{i\pi t}\langle A_{-}^tg,f\rangle$. 

For $t\in\mathbb{R}^+$, one has
\begin{equation}
\sum_{g\in\G}|\langle f_n,A^tg\rangle|^2=F_n(t)+G_n(t),
\end{equation}
where 
\begin{eqnarray*}
	F_n(t)&=&\sum_{g\in\G}(\langle A_{+}^tg^\Re,f_n^\Re\rangle+\langle A_{+}^tg^\Im,f_n^\Im\rangle+\cos(\pi t)\cdot(\langle A_{-}^tg^\Re,f_n^\Re\rangle+\langle A_{-}^tg^\Im,f_n^\Im\rangle)+\\
	&&\sin(\pi t)\cdot
	(\langle A_{-}^tg^\Re,f_n^\Im\rangle-\langle A_{-}^tg^\Im,f_n^\Re\rangle))^2,
\end{eqnarray*}
and
\begin{eqnarray*}
	G_n(t)&=&\sum_{g\in\G}(\langle A_{+}^tg^\Im,f_n^\Re\rangle-\langle A_{+}^tg^\Re,f_n^\Im\rangle+\sin(\pi t)\cdot(\langle A_{-}^tg^\Re,f_n^\Re\rangle+\langle A_{-}^tg^\Im,f_n^\Im\rangle)+\\
	&&\cos(\pi t)\cdot(\langle A_{-}^tg^\Im,f_n^\Re\rangle-\langle A_{-}^tg^\Re,f_n^\Im\rangle))^2.
\end{eqnarray*}
Note that for  $t \in \Omega_{L/4}\cap D_r$, by Lemma \ref{ReImBessel},  one has
\begin{eqnarray*}
	|F_n(t)|&\leq&6\cdot\left(\sum_{g\in\G}|	\langle f_n^\Re,A_{+}^tg^\Re\rangle|^2+|\langle f_n^\Im,A_{+}^tg^\Im\rangle|^2+\frac{3+e^{2\pi r}}{4}\cdot(|\langle f_n^\Re,A_{-}^tg^\Re\rangle|^2+\right.\\
	&&\left.|\langle f_n^\Im,A_{-}^tg^\Im\rangle|^2)+\frac{3+e^{2\pi r}}{4}\cdot(|\langle f_n^\Re,A_{-}^tg^\Im\rangle|^2+|\langle f_n^\Im,A_{-}^tg^\Re\rangle|^2)\right)\\
	&\leq&6\cdot\left(C_\G\|A\|^{2r}+\frac{3+e^{2\pi r}}{4}\cdot C_\G\|A\|^{2r}+\frac{3+e^{2\pi r}}{4}\cdot C_{\G}\|A\|^{2r}\right)\\
	&=&(15+3e^{2\pi r})\cdot C_\G\cdot\|A\|^{2r},
\end{eqnarray*} 
and
\begin{eqnarray*}
	|G_n(t)|&\leq& (15+3e^{2\pi r})\cdot C_\G\cdot\|A\|^{2r}.
\end{eqnarray*}

Thus, (using a similar proof as in Lemma \ref{AnalPosA}) $F_n$ and $G_n$ are uniformly bounded  analytic functions in $\Omega_{L/4}\cap D_r$.

As in Case 1, one can find two subsequences  $\{F_{n_k}\}$ and $\{G_{n_k}\}$ converging to analytic functions $F$ and $G$, respectively. Moreover, since $G_n(t)\leq \sum_{g\in\G}|\langle f_n,A^tg\rangle|^2$, and  $F_n(t)\leq \sum_{g\in\G}|\langle f_n,A^tg\rangle|^2$ for all $ t\in \R^+$ , and  $\lim_{n\rightarrow\infty}\sum_{g\in\G}|\langle f_n,A^tg\rangle|^2=0, a.e.~ t\in[0,L],$ one can proceed as in the proof of Case 1 and get the contradiction that
$$\lim_{n_{k_j}\rightarrow\infty}\sum_{g\in\G}\int_{0}^{1}|\langle f_{n_{k_j}},A^tg\rangle|^2=0.$$
Thus,  $\{A^tg\}_{g\in\G,t\in[0,L]}$ is a semi-continuous frame for $\HH$.
\end{proof}

\section{Acknowledgements}
The authors would like to thank Mr. Bin Sun for stimulating discussions, and Dr. Keaton Hamm and  Dr. Tao Wang for helping us to improve the manuscript. We also would like to thank Dr. Rudy Rodsphon for helpful suggestions.

\bibliographystyle{plain}
\bibliography{Akram_refs}

\end{document}